\documentclass[11pt]{amsart}     
\usepackage{graphicx}
\usepackage{times}      
\usepackage{latexsym}   
\usepackage{amssymb}    
\usepackage{amsmath}    
\usepackage{amsbsy}
\usepackage{amsthm}
\usepackage{amsgen}
\usepackage{amsfonts}
\usepackage{array}
\usepackage{epsfig}
\usepackage{hyperref}
\usepackage{color}
\usepackage{verbatim}
\usepackage{psfrag}

\newtheorem{theorem}{Theorem}[section]

\newtheorem{isotropylemma}[theorem]{Isotropy Lemma}
\newtheorem{doublesoulthm}[theorem]{Double Soul Theorem}
\newtheorem{soulthm}[theorem]{Soul Theorem}
\newtheorem{xtlemma}[theorem]{Extent Lemma}

\newtheorem*{thma}{Theorem A}

\newtheorem{corollary}[theorem]{Corollary}

\newtheorem{lemma}[theorem]{Lemma} 
\newtheorem{proposition}[theorem]{Proposition}

\newtheorem{proof-ref}{Proof(Proposition{P:5t2})}
\def\lpq{L_{p, q}}

\def\rrr{\mathbb{R}}
\def\ccc{\mathbb{C}}
\def\zzz{\mathbb{Z}}

\DeclareMathOperator{\Fix}{Fix}
\def\bdm{\begin{displaymath}}
\def\edm{\end{displaymath}}
\def\beq{\begin{equation}}
\def\eeq{\end{equation}}
\def\bes{\begin{equation*}}
\def\ees{\end{equation*}}
\def\epcm{\end{picture}\end{center}\end{minipage}}
\def\bpcm{\begin{minipage}{80pt}\begin{center}\begin{picture}}
\def\Mb{\textrm{Mb}}
\def\t2{T^2}

\def\f4{F_4}
\def\g2{G_2}

\def\p2{\frac{\pi}{2}}
\def\dist{\textrm{dist}}

\def\Fix{\textrm{Fix}}
\def\txt{\textrm}

\def\Kl{\txt{Kl}}
\def\d{\partial}
\def\dim{\textrm{dim}}
\def\ra{\rightarrow}

\DeclareMathOperator{\Cohomfix}{cohomfix}
\DeclareMathOperator{\curv}{curv}
\DeclareMathOperator{\xt}{xt}
\DeclareMathOperator{\rk}{rk}

\theoremstyle{definition}

\newtheorem{example}[theorem]{Example}
\newtheorem*{ack}{Acknowledgements}
\newtheorem{remark}[theorem]{Remark}

\begin{document}


\title[Non-negatively curved $5$-manifolds with almost maximal symmetry rank]{Non-negatively curved ${\bf 5}$-manifolds with almost maximal symmetry rank
	}






\author[F. Galaz-Garcia]{Fernando Galaz-Garcia}
\address[F. Galaz-Garcia]{Mathematisches Institut, Westf\"alische Wilhelms-Universit\"at M\"unster, GERMANY.}
\email[]{f.galaz-garcia@uni-muenster.de}

\author[C. Searle]{Catherine Searle}
\address[C. Searle]{Avenida Domingo Diez 111-13, Colonia Miraval, Cuernavaca, MEXICO.}
\email[]{searle.catherine@gmail.com}


\keywords{Almost, Maximal, Symmetry, Rank, Non-negative, Curvature, 5-manifold, Torus, Action}
\subjclass[2000]{53C20, 57S25, 51M25}

\date{\today}


\begin{abstract}

We show that a closed, simply-connected, non-negatively curved  $5$-man\-ifold admitting an effective, isometric $T^2$ action
is diffeomorphic to one of $S^5, S^3\times S^2$, $S^3\tilde{\times} S^2$ or the Wu manifold $SU(3)/SO(3)$.

\end{abstract}

\maketitle


\section{Introduction}
\label{S:intro}

The classification of Riemannian manifolds with positive, and more generally, non-negative sectional curvature, is a long-standing open problem in Riemannian geometry. As a step towards  general classification  results one may consider manifolds whose isometry group is large. This has been a fruitful avenue of research (see, for example, the surveys \cite{Gr,Wi2,Z,Gr2}). It is well-known that the isometry group of a compact Riemannian manifold is a compact Lie group. In the context of this paper, the measure for the ``size" of  an isometry group is its rank. In particular, we are interested in manifolds with non-negative curvature that have almost maximal symmetry rank, where the \emph{symmetry rank} of a Riemannian manifold $M$ is defined to be the rank of the isometry group of $M$. 

Grove and Searle \cite{GS} showed that the symmetry rank of a closed, positively curved, Riemannian $n$-manifold is bounded above by $[(n+1)/2]$ and classified closed, positively curved Riemannian manifolds with maximal symmetry rank  up to diffeomorphism. For a closed, positively curved Riemannian $n$-manifold of almost maximal symmetry rank, that is, one whose isometry group has rank $[(n-1)/2]$, Rong \cite{R}  found topological restrictions for all dimensions (distinguishing the cases for even and odd) and showed that a closed, simply-connected, positively curved Riemannian $5$-manifold with almost maximal symmetry rank, that is, with an effective isometric $T^2$ action, must be homeomorphic to the $5$-sphere (in fact, it will be diffeomorphic, as a consequence of the Generalized Poincar\'e conjecture). Later, Wilking \cite{Wi} improved these results significantly for closed, positively curved, simply-connected $n$-manifolds of dimension $n\geq 10$, considering actions of rank approximately $n/4$.

The maximal symmetry rank for closed, simply-connected $n$-manifolds with non-negative curvature and dimension $n\leq 9$ is $[2n/3]$ (see \cite{GGS}). Kleiner \cite{K} and Searle and Yang \cite{SY94}  independently 
classified up to homeomorphism closed, simply-connected $4$-manifolds of non-negative curvature with an effective isometric circle action, corresponding to the almost maximal symmetry rank case in dimension $4$.  In \cite{GGS}, the authors classified up to diffeomorphism closed, simply-connected, non-negatively curved Riemannian manifolds 
of dimensions $3$, $4$, $5$ and $6$ with maximal symmetry rank. In this paper we address the case of almost maximal symmetry rank  for closed, simply-connected, non-negatively curved Riemannian manifolds in dimensions $3$, $4$ and $5$. Our main result is the following:


\begin{thma}\label{T:A} Let $M^5$ be a closed, simply-connected, non-negatively curved  $5$-man\-ifold. If $T^{2}$ acts isometrically and (almost) effectively on
$M^5$, then $M^5$ is diffeomorphic to one of $S^5$, $S^3\times S^2$, $S^3\tilde{\times} S^2$ (the non-trivial $S^3$-bundle over $S^2$) or the Wu manifold $SU(3)/SO(3)$.
\end{thma}

We remark that the $5$-manifolds listed  in Theorem~A are all the known examples of closed, simply-connected $5$-manifolds with non-negative  curvature and these manifolds are all the closed, simply-connected $5$-dimensional homogeneous spaces or biquotients of Lie groups (cf. \cite{DV,Pv}). We also point out that the $5$-manifolds listed in Theorem~A  coincide with the closed, simply-connected $5$-manifolds that are elliptic (cf. \cite{PP}).  Further, each one of these $5$-
manifolds $M$ is rationally elliptic,  that is, $\dim\,\pi_{*}(M)\otimes \mathbb{Q} < \infty$, thus satisfying the Ellipticity Conjecture, which states that all closed, simply-connected manifolds of (almost) non-negative curvature are rationally elliptic (cf. \cite{Gr}). It is also worth noting that these are exactly the $5$-dimensional topological manifolds $M$ for which 
${\textrm{cat}}_{S^2}(M)=2$, that is, $M$ can be covered by two open subsets $W_1$, $W_2$ such that the inclusions \linebreak $W_i\hookrightarrow M$ factor homotopically through maps $W_i\rightarrow S^2$  (cf. \cite{GGH}).

This paper is divided into seven sections. The first two sections comprise the introduction and 
basic tools we will use throughout. In section 3, using classification results for smooth circle actions on $3$- and $4$-manifolds, in combination with restrictions imposed by non-negative curvature, we classify closed, orientable manifolds with non-negative curvature and almost maximal symmetry rank in dimension $3$ and recall the classification of closed, simply connected manifolds with non-negative curvature and almost maximal symmetry rank in dimension $4$.  In section 4 we consider the problem of almost maximal symmetry rank in dimension $5$ from a purely topological perspective and in section 5 we find restrictions imposed by non-negative curvature. In section 6 we classify closed, simply-connected, non-negatively curved $5$-manifolds of almost maximal symmetry rank by applying the results of the previous three sections. Finally, in section~7 
we give examples of actions of almost maximal symmetry rank on some of the manifolds listed in Theorem~A.

\begin{ack}
The authors thank  both Karsten Grove and Burkhard Wilking for pointing out an omission in an earlier version and for helpful  conversations and suggestions 
for improvements. Both authors would also like to thank 
the American Institute of Mathematics (AIM)  for its support during a workshop where the work on this paper was initiated.  F. G. G. also thanks the Mathematics Department of the University of Maryland, for its financial support,  and the Department of Mathematics of the University of Notre Dame, as well as the IMATE-UNAM, Cuernavaca, for their hospitality while part of this work was carried out.  C. S. was supported in part by CONACyT Project \#SEP-106923, CONACyT Project \#SEP-82471.
\end{ack}

\section{Definitions and tools}\label{S:1}

In this section we gather several definitions and results that we will use in subsequent sections. 


\subsection{Transformation groups}
Let $G$ be a Lie group acting (on the left) on a smooth manifold $M$.  We denote by $G_x=\{\, g\in G : gx=x\, \}$ the \emph{isotropy group} at $x\in M$ and by $Gx=\{\, gx : g\in G\, \}\simeq G/G_x$ the \emph{orbit} of $x$. The \emph{ineffective kernel} of the action is the subgroup $K=\cap_{x\in M}G_x$. We say that $G$ acts \emph{effectively} on $M$ if $K$ is trivial. The action is called \emph{almost effective} if $K$ is finite. The action is \emph{free} if every isotropy group is trivial and \emph{almost free} if every isotropy group is finite. 
We will denote the \emph{fixed point set} $M^G=\{\, x\in M : gx=x, g\in G \, \}$ of this action by $\Fix(M ; G )$ and define its dimension as $\dim(\Fix(M;G))=\max \{\,\dim(N): \text{$N$ is a component of $\Fix(M;G)$}\,\}$. When convenient, we will denote the orbit space $M/G$ by $X$. We will denote by $\overline{p}$ the image  of a point $p\in M$ under the orbit  projection map $\pi:M\rightarrow M/G$. Given a subset $A\subset M$, we will denote its image in $X$ under the orbit projection map by $A^*$ and when convenient, we shall also denote the orbit space $M/G$ by $M^*$.

 One measurement for the size of a transformation group $G\times M\rightarrow M$ is the dimension of its orbit space $M/G$, also called the {\it cohomogeneity} of the action. This dimension is clearly constrained by the dimension of the fixed point set $M^G$  of $G$ in $M$. In fact, $\dim (M/G)\geq \dim(M^G) +1$ for any non-trivial action. In light of this, the {\it fixed-point cohomogeneity} of an action, denoted by $\Cohomfix(M;G)$, is defined by
\[
\textrm{cohomfix}(M; G) = \dim(M/G) - \dim(M^G) -1\geq 0.
\]
A manifold with fixed-point cohomogeneity $0$ is also called a {\it fixed point homogeneous manifold}. We will use the latter term throughout this article. We observe that the fixed point set of a fixed point homogeneous action has codimension 1 in the orbit space.

\begin{remark}
Throughout  the rest of the paper we will assume all manifolds to be smooth. We will only consider smooth (almost) effective actions and all homology and cohomology groups will have coefficients in $\zzz$, unless otherwise stated. 
\end{remark}


\subsection{Alexandrov geometry} Recall that a finite dimensional length space $(X,\mathrm{dist})$ is an \emph{Alexandrov space} if it has curvature bounded from below (cf. \cite{BBI}).
 When $M$ is a complete, connected Riemannian manifold and $G$ is a compact Lie group acting on $M$ by isometries, the orbit space $X=M/G$ is equipped with the orbital distance metric induced from $M$, that is, the distance between $\overline{p}$ and $\overline{q}$ in $X$ is the distance between the orbits $Gp$ and $Gq$ as subsets of $M$.  If, in addition, $M$ has sectional curvature bounded below, that is, $\sec M\geq k$, then the orbit space $X$ is an Alexandrov space with $\curv X \geq k$. 
 
The \emph{space of directions} of a general Alexandrov space at a point $x$ is
by definition the completion of the 
space of geodesic directions at $x$. In the case of orbit spaces $X=M/G$, the space of directions $\Sigma_{\overline{p}}X$ at a point $\overline{p}\in X$ consists of geodesic directions and is isometric to
\[
S^{\perp}_p/G_p,
\] 
where $S^{\perp}_p$ is the unit normal sphere to the orbit $Gp$ at $p\in M$.

We now state Kleiner's Isotropy Lemma (cf. \cite{K}), which we will use to obtain information on the distribution of the isotropy groups along minimal geodesics  joining two orbits and, in consequence, along minimal geodesics joining two points in the orbit space $X=M/G$.


\begin{isotropylemma}
\label{l:Kleiner}
 Let $c:[0,d]\rightarrow M$ be a
minimal geodesic between the orbits $G c(0)$ and $G c(d)$. Then, for
any $t\in (0,d)$, $G_{c(t)}=G_{c}$ is a subgroup of $G_{c(0)}$ and
of $G_{c(d)}$.
\end{isotropylemma}

Recall that the $q$-\emph{extent} $\xt_q(X)$, $q\geq 2$, of a compact metric space $(X,d)$ is the maximum average distance between $q$ points in $X$:
\[
\xt_q(X)=\binom{q}{2}^{-1}\max\{\, \sum_{1\leq i<j\leq q}d(x_i,x_j):\{x_i\}_{i=1}^{n}\subset X\,\}.
\] 

The Extent Lemma (cf. \cite{GS2}) stated below provides an upper bound on the total number of isolated 
singular points in $X=M/G$.


\begin{xtlemma}\label{L:xt} Let $\overline{p}_0, \ldots, \overline{p}_q$ be $q+1$ distinct points in  $X=M/G$. If $\curv X\geq 0$, then
\[
\frac{1}{q+1}\sum_{i=0}^{q}\xt_q (\Sigma_{\overline{p}_i}X) \geq\pi/3.
\] 
\end{xtlemma} 
We remark that in the case of strictly positive curvature, the inequality is also strict.

We will also use the following analogue for orbit spaces of the Cheeger-Gromoll Soul Theorem to obtain information on the geometry of the orbit space $X=M/G$.  


\begin{soulthm} \label{T:soulthm} Let $X=M/G$. If $\curv X\geq 0$ and $\partial X\neq
\emptyset$, then there exists a totally convex compact subset
$S\subset X$ with $\partial S=\emptyset$, which is a strong
deformation retract of $X$. If $\curv M/G>0$, then $S=\overline{s}$ is a
point, and $\partial X$ is homeomorphic to $\Sigma_{\overline{s}}X\simeq
S_{s}^{\perp}/G_s$.
\end{soulthm}

When $M$ is a non-negatively curved, fixed point homogeneous Riemannian $G$-manifold, the orbit space $X$ is a non-negatively curved Alexandrov space and $\partial X$ contains a component $N$ of $\Fix(M;G)$. Let $C\subset X$ denote the set at maximal distance from $N\subset \partial X$ and let $B=\pi^{-1}(C)$. The Soul Theorem~\ref{T:soulthm} implies that $M$ can be written as the union of neighborhoods $D(N)$ and $D(B)$ along their common boundary $E$, that is, 
\[
M=D(N)\cup_E D(B).
\]
In particular, when $G=T^1$ and $C$ is another fixed point set component with maximal dimension, one has the following result from \cite{SY94}.


\begin{doublesoulthm}\label{T:sy2}  Let $M$ be a complete, non-negatively curved Riemannian manifold admitting an isometric $T^1$ action. If $\Fix(M;T^1)$ contains two codimension $2$ components $N_1$ and $N_2$, with one of them being compact, then $N_1$ is isometric to $N_2$,  $\Fix(M; T^1) = N_1\cup N_2$, and $M$ is diffeomorphic to an $S^2$-bundle over $N_1$ with $T^1$ as its structure group. In other words, there is a principal $T^1$-bundle, $P$, over $N_1$ such that $M$ is diffeomorphic to $P\times_{T^1} S^2$.
\end{doublesoulthm}


\subsection{Closed $3$-manifolds with a smooth $T^2$ action}
\label{SS:3MT2}

We recall the list of closed $3$-manifolds with a smooth cohomogeneity one $T^2$ action (cf. \cite{M,N}), as they will appear throughout the paper.  They are:  $S^3$, a lens space $\lpq$, $S^2\times S^1$, $\rrr P^2\times S^1$, $T^3$, $S^2\tilde{\times} S^1$, $\Kl\times S^1$ and $A$. Here $\Kl$ denotes the $2$-dimensional Klein bottle and $S^2\tilde{\times} S^1$ the non-trivial $S^2$-bundle over $S^1$. The manifold $A$ is obtained by gluing $\Mb\times S^1$ and $S^1\times \Mb$ along their boundary torus, where $\Mb$ denotes the M\"obius band.


\section{Non-negatively curved 3- and 4-manifolds with almost maximal symmetry rank}\label{S:2}

In this section we classify closed, orientable $3$-manifolds and closed, simply-connected $4$-manifolds, assuming they have non-negative curvature and admit  an isometric action of a circle $T^1$.

 
\subsection{Dimension 3}

In the case of a $T^1$ action, we have the following result, which follows from the Orlik-Raymond-Seifert classification of $3$-manifolds with a \linebreak
smooth $T^1$ action \cite{Or,OR}.


\begin{theorem}\label{T:n3.1}
Let $T^1$ act isometrically on $M^3$, a closed, orientable $3$-manifold of non-negative curvature. Then 
$M^3$ is equivariantly diffeomorphic to a spherical $3$-manifold, $S^2\times S^1$,   $\rrr P^3\# \rrr P^3$, $T^3$ or one of four $T^2$-bundles over $S^1$.
\end{theorem}


\begin{proof} We break the proof into three cases: case 1, where the action is free, case 2, where the action is almost free, and case 3, where the action has non-trivial fixed point set.


\vskip .5cm
\noindent\emph{Case 1: $ T^1$ acts freely}
\vskip .25cm
In this case $X^2=M^3/T^1$ is a closed, orientable  $2$-manifold of non-negative curvature and thus $X^2=S^2$ or $X^2=T^2$ by the Gauss-Bonnet theorem. Since the action is free, $M^3$ is a principal circle bundle over $X^2$ and therefore
$M^3$ is diffeomorphic to one of $S^3$, $\lpq$, $S^2\times S^1$ or $T^3$.


\vskip .5cm
\noindent\emph{Case 2: $T^1$ acts almost freely}
\vskip .25cm
Here $M^3$ is a Seifert manifold supporting a smooth circle action. Since we have assumed that $M^3$ has non-negative curvature, $M^3$ admits a geometric structure modeled on  $S^3$, $S^2\times \rrr$ or Euclidean space $E^3$ (cf. \cite{Sc}). Closed, orientable Seifert manifolds with $S^3$, $E^3$ or $S^2\times \rrr$ geometry supporting a smooth $T^1$ action have been classified (cf. \cite{Sc,Or}). When $M^3$ has $S^3$ geometry, $M^3$ must be diffeomorphic to a spherical $3$-manifold, that is, a quotient of $S^3$ by a finite subgroup of $SO(4)$ acting freely on $S^3$. 
We denote these manifolds in the usual fashion by their $2$-dimensional orbit spaces. The $3$-sphere $S^3$ is denoted by $S^2$ and $\lpq$ by $S^2(p)$ or $S^2(p, q)$. The remaining manifolds with $S^3$ geometry are denoted by $S^2(2, 2, n)$, $S^2(2, 3, 3)$, $S^2(2, 3, 4)$ and $S^2(2, 3, 5)$, where $n\geq 2$ is an integer.

 When $M^3$ has $S^2\times \rrr$ geometry, $M^3$ must be $S^2\times S^1$ and, when $M^3$ has $E^3$ geometry, it must be diffeomorphic to $T^3$ or to one of four of the remaining five possible orientable, closed flat manifolds covered by $T^3$. The fifth possibility is excluded immediately since it does not admit a circle action. These four flat manifolds covered by $T^3$ are $T^2$ bundles over $S^1$, described in \cite{Or}. Their orbit spaces are $S^2(2, 2, 2, 2)$, $S^2(2, 4, 4)$, $S^2(2, 3, 6)$ and $S^2(3, 3, 3)$. Further, all of these closed, orientable $3$-manifolds with $E^3$, $S^3$ or $S^2\times \rrr$ geometry, with a Seifert fibration induced by an almost free circle action, do admit isometric circle actions inducing the given Seifert fibration (cf. \cite{Sc,Or}). 


\vskip .5cm
\noindent\emph{Case 3: $T^1$ has non-trivial fixed point set}
\vskip .25cm


By definition, the action is fixed point homogeneous. Closed fixed point homogeneous manifolds $3$-manifold with nonnegative curvature were classified in \cite{GG} and we recall their classification in the orientable case.  Observe first that the fixed point set is $1$-dimensional, with at most two components, and these components are circles. If $\Fix(M^3;S^1)$ contains two components, then by the Double Soul Theorem \ref{T:sy2} we see that $M^3$ is 
one of the two $S^2$ bundles over $S^1$ and since $M^3$ is assumed to be orientable, it must be $S^2\times S^1$. If $\Fix(M^3;S^1)$ consists of a single component $F^1$, then  $X^2=M^3/S^1$ is a $2$-dimensional Alexandrov space of non-negative curvature with boundary $F^1\cong S^1$. Thus $X^2$ is an orientable, non-negatively curved topological manifold with boundary  and the only possibilities are  $D^2$ and $S^1\times I$.  We may exclude $S^1\times I$ since $M^3$ is assumed to be orientable. Thus $D^2$ is the only possible orbit space. 
The non-negative curvature hypothesis implies that  the interior of $D^2$ has either no points with non-trivial finite isotropy, one point with finite isotropy $\zzz_p$ or two points with finite isotropy $\zzz_2$. These correspond, respectively, to $S^3$, a lens space $\lpq$ and $\rrr P^3\# \rrr P^3$. 
\\

It follows from the three cases analyzed above that $M^3$ can only be  $S^3/\Gamma$, where $\Gamma$ is a finite subgroup of $SO(4)$, $ S^2\times S^1$, $T^3$, one of the four flat $T^2$-bundles over $S^1$ covered by $T^3$  or, finally, $ \rrr P^3\# \rrr P^3$. Each of these manifolds supports only one isometric $T^1$ action with non-negative curvature yielding the possible orbit space structures (cf. \cite{Ra}).
\end{proof}

 
\subsection{Dimension 4}

Given Perelman's work on the Poincar\'e conjecture \cite{P1,P2,P3}, the classification of closed, simply-connected, non-negatively curved $4$-man\-i\-folds admitting an  isometric  $T^1$ action follows from earlier classification results in a curvature-free setting and a restriction on the Euler characteristic, which is a simple consequence of the non-negative curvature assumption. 
The first theorem is due to work of Fintushel \cite{F1,F2} in combination with work of Pao \cite{Pa} and Perelman's proof of the Poincar\'e conjecture \cite{P1,P2,P3}.


\begin{theorem}\label{T:FPP} A closed, simply-connected smooth $4$-manifold with a $T^1$ action is equivariantly diffeomorphic to a connected sum of $S^4$, $\pm\ccc P^2$ and $S^2\times S^2$.
\end{theorem}

Let $M^4$ be a closed, simply-connected, non-negatively curved $4$-manifold and let $\chi(M^4)$ be its Euler characteristic. It follows from work done independently by Kleiner \cite{K} and Searle and Yang \cite{SY94} that $2\leq \chi(M^4)\leq 4$. Combining this with theorem~\ref{T:FPP} yields the following result in the case of non-negative curvature (cf. \cite{GG}).


\begin{theorem}\label{T:sy1}  A closed, simply-connected, non-negatively curved $4$-manifold with an isometric  $T^1$ action is diffeomorphic to $S^4, \ccc P^2, S^2\times S^2$ or $\ccc P^2\#\pm \ccc P^2$.
\end{theorem}


\section{Cohomogeneity three torus actions on  simply-connected $5$-manifolds}

Here we gather general facts about smooth cohomogeneity three torus actions $T^n\times M^{n+3}\rightarrow M^{n+3}$ on simply-connected, smooth manifolds and then consider the specific case when $M$ is $5$-dimensional. 
The main goal of this section is to understand the structure of the singular sets, that is, the set of points  in the orbit space $M^*$ corresponding to orbits with non-principal isotropy groups.


\subsection{General considerations}

We begin with the following theorem from Bredon \cite{Br}, which characterizes the orbit space of a cohomogeneity three action.


\begin{theorem}\label{t:Bredon} Let $G$ be a compact Lie group acting by cohomogeneity three on $M$, a compact, simply-connected smooth manifold. If all orbits are connected, then $M^*$ is a simply-connected topological $3$-manifold with or without boundary.
\end{theorem}

It follows from the resolution of the Poincar\'e conjecture (cf. \cite{P1,P2,P3}) that $M^*$ is homeomorphic to one of $S^3$, $D^3$, $S^2\times I$ or, more generally, to $S^3$
 with a finite number of disjoint open $3$-balls removed.  
We will see in the next section that non-negative curvature implies that 
$M^*$ can only be one of the first three manifolds from this list.

We also recall the following general result of Bredon \cite{Br} about the fundamental group of the orbit space:

\begin{theorem}\label{t:Bredonpi1} Let  $G$ be a compact Lie group acting on a topological space $X$. If either $G$ is connected or $G$ has a nonempty fixed point set, then the orbit projection map $\pi: X\rightarrow X/G$ induces an onto map  on fundamental groups.
\end{theorem}

The next theorem, again from Bredon~\cite{Br}, implies the absence of special exceptional orbits and, in particular, allows us to conclude that no fixed point set of finite $\zzz_2$-isotropy has 
codimension one in $M$. This result will be used in the proof of lemma \ref{l:nofreeoralmostfree}.


\begin{theorem}\label{T:Bredon_NSE}  Let $M$ be a smooth, simply-connected manifold admitting an action by a compact Lie group. If a principal orbit is connected (and hence all orbits are connected) then there are no special exceptional orbits, that is, the dimension of the set of points belonging to exceptional orbits is of codimension greater than or equal to $2$.
\end{theorem}


\begin{lemma}\label{l:nofreeoralmostfree} Let $T^n$ act on $M^{n+3}$, a closed, simply-connected smooth manifold. Then some circle subgroup has non-trivial fixed point set.
\end{lemma}


\begin{proof}
 If all circle subgroups were to act freely, this would imply a free circle action on a closed, simply-connected $4$-manifold $M^4=M^{n+3}/T^{n-1}$, which is impossible. Likewise, if the action is almost free, then there are finitely many finite isotropy groups.
Let $\Gamma$ be the finite group generated by all these finite groups and consider the action of $T^n/\Gamma$ on $M^{n+3}/\Gamma$, a closed, simply-connected topological space by theorem \ref{t:Bredonpi1}. We claim that $M^{n+3}/\Gamma$ must be a topological manifold. Note that the fixed point set of any subgroup of finite isotropy must be at least $n$-dimensional since it is invariant under the $T^n$ action and it will be at most $(n+1)$-dimensional because there are no special exceptional orbits by theorem  \ref{T:Bredon_NSE}.  
The  space of directions to the projection of these fixed point sets in $M^{n+3}/\Gamma$ will be either topological $2$-spheres or circles in all but one case, that is, where $\zzz_2$ fixes an $n$-dimensional submanifold and acts freely on its normal $S^2$. 
However, this latter case is impossible, because such a fixed point set would project to a point in $M^{n+3}/T^n$ with space of directions $\rrr P^2$ and, by theorem \ref{t:Bredon}, this is a contradiction. Hence  $M^{n+3}/\Gamma$ must be a closed, simply-connected topological manifold. Now, $T^n/\Gamma$ must act freely on $M^{n+3}/\Gamma$ and we have just seen that this is impossible. Therefore $T^n$ cannot act almost freely on  $M^{n+3}$ either.
 \end{proof}

Let $M^{n+3}$ be a closed, simply-connected $(n+3)$-manifold with a cohomogeneity three $T^n$ action. By the previous lemma, there is a circle subgroup $T^1 \subset T^n$ with non-trivial fixed point set. In the case where $M^*=D^3$, there is a unique codimension $2$ fixed point set component. In general, when $M^*$ is homeomorphic to $S^3$ with $k$  
disjoint open $3$-balls removed, $k\geq 1$, there will be $k$ codimension $2$ components of the fixed point sets of possibly different circles, one for each boundary component of $M^*$.

 In the case where $M^*$ is homeomorphic to $S^3$, the components of $\Fix(M^{n+3}; T^1)$ are of codimension greater than or equal to $4$. In this case, the following proposition, generalizing a result of Rong's in dimension $5$ (cf. \cite{R}), shows there must be a minumum number of codimension $4$ fixed point set components, corresponding to isolated singular orbits $T^{n-1}$.


\begin{proposition}\label{p:Tn-1orbits}
Let $T^n$ act on $M^{n+3}$, a simply-connected, smooth manifold. Suppose that $M^*$ is homeomorphic to $S^3$ and that there are exactly two orbit types: principal orbits $T^n$ and isolated singular orbits $T^{n-1}$. Then there are at least $n+1$ isolated singular orbits $T^{n-1}$.
\end{proposition}


\begin{proof}
Let $M_0$ denote the manifold with boundary obtained by removing a small tubular neighborhood around each isolated singular orbit $T^{n-1}$. Let $M^*_0$ denote the quotient space $M_0/T^2$. By a standard transversality argument we know that
\bdm
\pi_1(M_0)\cong \pi_1(M) = \{1\}
\edm
and  
\bdm
\pi_2(M_0)\cong \pi_2(M).
\edm

 Since there is no isotropy group of finite order we obtain a fibration 
 $$
 T^n\rightarrow M_0\rightarrow M^*_0,
 $$
 and therefore a long exact sequence in homotopy:
\bdm
0\rightarrow\pi_2(M_0)\rightarrow\pi_2(M^*_0)\rightarrow\pi_1(T^n)\rightarrow\pi_1(M_0)\rightarrow\pi_1(M^*_0)\rightarrow 0. 
\edm
Since $\pi_1(M)\cong \pi_1(M_0)=0$, it follows that $\pi_1(M^*_0)=0$. Since $M^*$ is a 3-sphere, applying the Mayer-Vietoris sequence to the pair 
$(M^*_0, \textrm{cl}(M^* \backslash M^*_0))$, we obtain that $H_2(M^*_0)\cong \zzz^r$, where $(r+1)$ is the number of isolated singular orbits. It follows from the Hurewicz isomorphism  that $\pi_2(M^*_0) \cong H_2(M^*_0) \cong \zzz^r$ and the above exact sequence in homotopy becomes
\bdm
0\rightarrow\pi_2(M_0)\rightarrow\zzz^r\rightarrow\zzz^n\rightarrow 0.
\edm
We conclude that $n\leq r$ and thus there are at least $n+1$ isolated singular orbits.

\end{proof}


\begin{corollary}\label{c:finiteisotropy}
Proposition~\ref{p:Tn-1orbits} remains valid in the presence of 
finite isotropy.
\end{corollary}


\begin{proof} 

Let $\Gamma$ be the finite group generated by the finite isotropy groups of the action. As we saw earlier in the proof of lemma \ref{l:nofreeoralmostfree}, $M^n/\Gamma$ is a closed, topological manifold. Moreover,  $M^n/\Gamma$ is simply-connected by theorem \ref{t:Bredonpi1}.  Finally, observe that $T^n/\Gamma$ acts without finite isotropy on $M^{n+3}/\Gamma$ and the isolated $T^{n-1}/(T^{n-1}\cap \Gamma)$ orbits in $M^{n+3}/\Gamma$ correspond to isolated $T^{n-1}$ orbits in $M^{n+3}$.

\end{proof}


\begin{remark}We observe that in the case where we have a $T^2$ action on $M^5$, proposition~\ref{p:Tn-1orbits} implies  that when $M^*=S^3$, there are at least three isolated circle orbits.
\end{remark}


\subsection{Possible isotropy groups}

In this subsection, we use the isotropy representation of the possible isotropy groups to better understand fixed point components of finite isotropy and their corresponding images in the orbit space $M^*$.

By a theorem of Chang and Skjelbred~\cite{CS}, components of $\Fix(M; \zzz_k)$ are \linebreak
smooth submanifolds. When $k\neq 2$ these components are orientable and of even codimension. If $k=2$, components of $\Fix(M;\zzz_2)$ may also be non-orientable and by  theorem~\ref{T:Bredon_NSE}, of codimension at least $2$. In the case of a smooth  $T^2$ action on a closed, simply-connected smooth $5$-manifold $M^5$,  components of $\Fix(M^5;\zzz_k)$ must be at least $2$-dimensional, as we saw in the proof of lemma \ref{l:nofreeoralmostfree}. An analysis of the isotropy representations will show that for all cases the components of $\Fix(M^5;\zzz_k)$ must be $3$-dimensional.



\begin{proposition}\label{p:3dimfixed} Let $T^2$ act smoothly on $M^5$, a closed, simply-connected smooth $5$-manifold. If $M^*=S^3$, then the following hold. 
\begin{itemize}
	\item[(1)]The singular orbits of the action are $T^1$ and $T^1/\zzz_k$, $k\in \zzz^+$.\\
	\item[(2)]The exceptional orbits are $T^2/\zzz_k$, $k\geq 2$, and $T^2/(\zzz_2\times \zzz_2)$.\\
	\item[(3)] In all cases where there is finite cyclic isotropy, the corresponding fixed point set of finite isotropy is of dimension $3$.
\end{itemize}
\end{proposition}


\begin{proof} 

Since we have assumed that $M^*$ is homeomorphic to $S^3$, there are no points with $T^2$ isotropy. 
Observe that the normal sphere at any point of an exceptional orbit will be of dimension two.
Thus the finite isotropy group of an exceptional orbit must be a subgroup of $SO(3)$ and of $T^2$. Hence the only possible finite isotropy groups are $\zzz_k$, $k\geq 2$, and $\zzz_2\times \zzz_2$. This proves parts (1) and (2).

Now we prove part (3). We first consider the singular orbits, observing that if we have a singular orbit of the form $T^1/\zzz_k$, then we have a $T^1\times \zzz_k$ action on the normal $3$-sphere to any point of the orbit. In particular, there will be a finite cyclic subgroup 
of order $k$ in $T^1\times\zzz_k$ fixing circles in this normal $3$-sphere and therefore 
this orbit is contained in a fixed point set of finite isotropy of dimension $3$.  If the singular orbit is $T^1$, then the action of the circle on the normal $S^3$ is either free or almost free. In the latter case,  a finite cyclic subgroup fixes a $3$-dimensional submanifold which contains the singular orbit.

We now consider the exceptional orbits. For a $T^2/\zzz_k$ orbit, $k\neq 2$, the $\zzz_k$ action on $S^2$ is never free and thus this exceptional orbit will be contained in a $3$-dimensional submanifold fixed by $\zzz_k$, $k\neq 2$. It remains to show that for the exceptional orbit $T^2/\zzz_2$, the $\zzz_2$ isotropy group also does not act freely on its normal $S^2$. This follows from the fact that the antipodal map, which reverses orientation, generates the only free $\zzz_2$ action on $S^2$ and it is not a subgroup of $SO(3)$.

Finally, we consider the exceptional orbit $T^2/(\zzz_2\times \zzz_2)$. The action of the isotropy subgroup, $\zzz_2\times \zzz_2$, on the normal $S^2$ produces a quotient space equal to the double right-angled spherical triangle with three vertices, each of which is fixed by a 
different $\zzz_2$ subgroup of $\zzz_2\times \zzz_2$. Each fixed vertex corresponds to a $3$-dimensional submanifold fixed by the corresponding $\zzz_2$ subgroup. For each  $T^2/(\zzz_2\times \zzz_2)$ orbit we will have exactly three such fixed point sets intersecting in this orbit.
Thus, we conclude that the fixed point set of a finite cyclic group is always of dimension $3$.

\end{proof}

\subsection{The singular sets in $M^*=S^3$}

We now determine the structure of the singular sets in the orbit space in the particular case when $M$ is $5$-dimensional and $M^*=S^3$. We first observe that the existence of simple closed curves in $M^*$ whose points correspond to orbits with non-trivial finite cyclic isotropy $\mathbb{Z}_k$ is ruled out by work of Montgomery and Yang (cf. \cite{MY} Lemma 2.3). Taking this into account,  the following proposition follows directly from the proof of proposition~\ref{p:3dimfixed}.


\begin{proposition}\label{p:singsets} Let $T^2$ act smoothly on $M^5$, a closed, simply-connected smooth $5$-manifold. If  $M^*=S^3$,  then the set of points in $M^*$ with non-trivial isotropy corresponds to a graph and the following hold:
\begin{itemize}
\item[(1)] The vertices of the graph correspond to isolated singular points and the edges correspond to points with finite cyclic isotropy. \\
\item[(2)] The pre-image of the closure of an edge  corresponds to a $3$-dimensional manifold fixed by a finite cyclic group admitting a $T^2$ action of cohomogeneity one.\\
\item[(3)] Isolated singular points will correspond to isolated circle orbits. \\
\item[(4)] Edges can belong  to graphs with vertices of degree $1$, $2$ or $3$, but only the latter type corresponds to a $T^2/(\zzz_2\times \zzz_2)$ exceptional orbit.

\end{itemize}
\end{proposition}


\begin{figure}
\centering
\includegraphics[scale=0.7]{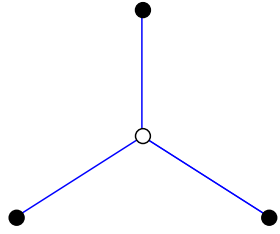}
\caption{Weighted claw: The central vertex has isotropy $\zzz_2\times\zzz_2$, the external vertices have isotropy conjugate to $S^1$ and the edges have isotropy $\zzz_2$.}
\label{F:triad}
\end{figure}

We will denote by \emph{arc}  the closure of an edge with finite cyclic isotropy in the set of orbits with non-trivial isotropy in $M^*$. Since the graphs corresponding to the singular set in $M^*$ carry isotropy information, we will refer to them as \emph{weighted graphs}.

We now begin the process of determining what $3$-manifolds may actually occur as fixed point set components of a finite cyclic isotropy group. Since these components admit an (almost) effective  $T^2$ action, 
they must be one of the manifolds listed in subsection \ref{SS:3MT2}. We will eventually show, in section \ref{s:dim5results},
that
the only such $3$-manifolds that can occur are  $S^3$, $\lpq$, $S^2\times S^1$ and $S^2\tilde{\times} S^1$.

We first observe that we may immediately rule out  $T^3$, since its orbit space would correspond to a simple closed curve in $M^*$ with finite cyclic isotropy and, as mentioned above, simple closed curves with finite cyclic isotropy will not occur. 

Of the possible $3$-manifolds on the list, the non-orientable ones are $\rrr P^2\times S^1$, $S^2\tilde{\times} S^1$, $\Kl\times S^1$ and $A$, and as such, they may only be fixed point set components of $\zzz_2$ isotropy.  All  have at least one exceptional orbit and correspond to the possible pre-images of arcs containing a vertex of degree three. 

If the singular set in $M^*$ contains a vertex of degree three, then it may contain different types of trees as subgraphs. Two types of trees may occur. The first type occurs if either $\rrr P^2\times S^1$ or $S^2\tilde{\times} S^1$ is the pre-image of an arc of $\zzz_2$ isotropy, in which case, the singular set contains a tree with one vertex of degree three joined to three vertices of degree one or two only. The second type occurs if $\Kl\times S^1$ or $A$  is the pre-image of an arc of $\zzz_2$ isotropy, in which case the singular set contains a tree with an edge terminating in two vertices of degree three, each of which is joined to two more vertices of degree one or two. We will see that when we take into consideration the lower curvature bound this second type of tree cannot occur, allowing us to exclude $\Kl\times S^1$ and $A$ as possible fixed point set components of $\zzz_2$ isotropy.
 
The first type of tree is the bipartite graph $K_{1,3}$, commonly known as a \emph{claw} (cf. \cite{D,GY}). Since vertices and edges carry isotropy information, we shall refer to this configuration as a \emph{weighted claw} (see Figure \ref{F:triad}).  An example of the second possible tree appears in Figure~\ref{F:2claws}. We will refer to such graphs as \emph{weighted trees}.
These graphs will appear in our analysis of the finite isotropy case in Section~\ref{s:5_dim_1}. 

Finally, we point out that the weighted graph could also contain a cycle. Moreover, this cycle could potentially be knotted in $M^*=S^3$. We will see in section~\ref{SS:unknot} that when the orbit space is non-negatively curved the  cycle cannot be knotted.


\begin{figure}
\centering
\includegraphics[scale=0.7]{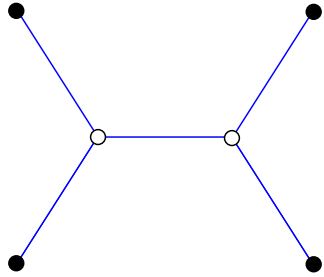}
\caption{Weighted tree: The two central vertices have isotropy $\zzz_2\times\zzz_2$, the external vertices have isotropy conjugate to 
$S^1$ and the edges have isotropy $\zzz_2$.}
\label{F:2claws}
\end{figure}

\section{Restrictions on the orbit space imposed by non-negative curvature}
\label{s:geom}

In this section we will see how non-negative curvature restricts the structure of the orbit space of an isometric $T^2$ action on a closed, simply-connected $5$-manifold. Throughout this section we will let $M^5$ be a closed, simply-connected $5$-manifold of non-negative curvature with an isometric $T^2$ action.


\subsection{Topology of orbit spaces with non-negative curvature}

As we noted earlier, the quotient space of a smooth $T^2$ action on a closed, simply-connected smooth $5$-manifold is homeomorphic to one of $S^3$ or $S^3$ with a finite number of disjoint open $3$-balls removed. For every open $3$-ball we remove we obtain an $S^2$ boundary component.  In the presence of non-negative curvature we have the following proposition.


\begin{proposition}\label{p:M*}
Let $M^5$ be a closed, simply-connected, non-negatively curved $5$-manifold. If $T^2$ acts isometrically on $M^5$, then $M^*$ is homeomorphic to one of the following:
\begin{itemize}
\item [(1)] $S^3$, if  for any $T^1\subset T^2$ for which $\Fix(M^5; T^1)\neq\emptyset$, $\dim(\Fix(M^5; T^1))=1$;\\

\item [(2)]  $D^3$ or $S^2\times I$, if $\dim(\Fix(M^5; T^1))=3$ for some $T^1\subset T^2$.
\end{itemize}
\end{proposition}


\begin{proof}

 Part (1) follows easily since only points belonging to a codimension two fixed point set of a circle will correspond to boundary points in the orbit space $M^*$. Note that part (1) is independent of the curvature assumption. Part (2) follows from the Double Soul theorem \ref{T:sy2}.

\end{proof}


\subsection{Upper bound on the number of isolated circle orbits in $M^5$}
\label{SS:4-circles}

In the previous section, in proposition~\ref{p:Tn-1orbits}, we found a lower bound of three for the number of isolated circle orbits in $M^5$ for the case where $M^*=S^3$. We now propose to determine an upper bound on the number of isolated circle orbits when $M^5$ is non-negatively curved. Theorem~\ref{T:4-circles} below will show that there can be at most four such orbits. 
 
 A simple application of the Extent lemma tells us that in $M^*=M/G$, where $G$ acts isometrically on $M$, a closed manifold of positive curvature, there are at most $3$ singular points with space of directions isometric to $S^2(1/2)$ or a ``thin'' $S^2(1/2)$, that is, the quotient of $S^3(1)$ by an almost free $S^1$ action. If $M$ is non-negatively curved, 
 the Extent lemma tells us that there will be at most $5$ such singular points. A closer analysis of the geometry will allow us to show that in the case where  $M$ is $5$-dimensional, non-negatively curved and admits an isometric $T^2$ action, there will be at most $4$ isolated circle orbits.

This upper bound follows from the generalization of an argument used in \linebreak Kleiner's thesis \cite{K}  showing that an isometric circle action on a closed, simply-connected, non-negatively curved $4$-manifold has at most four isolated fixed points.  We remark that the same result was obtained in \cite{SY94}, but the argument used to prove the result was specific to dimension $4$ and does not generalize to higher dimensions. The key observation that allows us to apply the techniques in \cite{K} to our situation  is that the normal sphere at a point to each one of the circles fixed by some $S^1\subset T^2$ is $3$-dimensional. We include the proof of the theorem here for the sake of completeness since Kleiner's result was never published. 


\begin{theorem}
\label{T:4-circles}  Let $M^5$ be a closed non-negatively curved $5$-manifold with an isometric $T^2$ action. Then there are at most 4 isolated circle orbits of the $T^2$ action.
\end{theorem}

The proof of theorem~\ref{T:4-circles} will occupy the remainder of this subsection. We begin by fixing some notation and recasting several lemmas from \cite{K} to meet our needs. We assume that all geodesics have unit speed unless stated otherwise. 

 Let $\{p_i\}_{i=1}^4$ be four distinct points in $M^5$. For $1\leq i,j\leq 4$, let $\Gamma_{ij}$ be the set of minimizing normal geodesics from $p_i$ to $p_j$ and, for each triple $1\leq i,j,k\leq 4$, let 
\[
\alpha_{ijk}=\angle_{p_i}(p_j, p_k) =\min\{\, \angle(\gamma'_j(0),\gamma'_k(0)) : \gamma_j\in \Gamma_{ij}, \gamma_k\in\Gamma_{ik}\,\}.
\]
For each pair of points $p_i$, $p_j$, $i\neq j$, let $\mathrm{Dir}_{p_i}(p_j)$ be the set of initial directions of a normal minimizing geodesic from $p_i$ to  $p_j$, that is, 
\[
\mathrm{Dir}_{p_i}(p_j)=\{\, \gamma'_j(0):\gamma_j\in \Gamma_{ij} \,\}.
\]
Finally, let $T_{ijkl}$ denote the (possibly degenerate) tetrahedron determined by the four points $p_i, p_j, p_k, p_l$.
\\


Before proceeding with the proof of theorem~\ref{T:4-circles}, we recall the  following fact (cf. \cite{SY94}).


\begin{lemma}\label{l:Lipschitz} Suppose $S^1$ acts isometrically and fixed point freely on $S^3(1)$.
Then $S^3/S^1$ is smaller than $S^2(1/2)=S^3/S^1_{\txt{Hopf}}$. That is, there 
is a surjective $1$-Lipschitz map $S^2(1/2)\rightarrow S^3/S^1$.
\end{lemma}

We have the following lemma.


\begin{lemma}

\label{L:angle_bound}
If there are $4$ isolated circle orbits $\{N_i\}_{i=1}^{4}$, then, for distinct points  $p_i\in N_i$\, , $1\leq i\leq 4$, and every quadruple of distinct integers $1\leq i,j,k, l\leq 4$, the tetrahedron $T_{ijkl}$ is rigid in the following sense:
\begin{equation}
\label{E:angle_bound}
\alpha_{ijk} + \alpha_{ijl} + \alpha_{ikl}  = \pi
\end{equation}
and
\begin{equation}
\label{E:side_bound}
\alpha_{ijk} + \alpha_{jki} + \alpha_{kij}  = \pi,
\end{equation}
that is, the sum of angles at each vertex and the sum of angles of each face of $T_{ijkl}$ are both $\pi$.

\end{lemma}


\begin{proof}
In the orbit space $X^3=M^5/T^2$, the $4$ circles $\{N_i\}_{i=1}^4$ correspond to 4 points $\{\bar{p}_i\}_{i=1}^4$. By Toponogov's theorem for Alexandrov spaces (cf. \cite{BGP}), we know that the sum of the angles of a geodesic triangle in $X^3$ will be greater than or equal to $\pi$. Connecting each pair of distinct  points in $\{\overline{p}_i\}_{i=1}^4$ by a minimal geodesic
we obtain a configuration of four triangles and the total sum of the angles in this configuration will be greater than or equal to $4\pi$.

For each one of the four points $\{p_i\}_{i=1}^4$ the coresponding isotropy group acts freely or almost freely on the normal space $T_{p_i}N_i^{\perp}$  and
the quotient of the unit normal sphere $S^3\subset T_{p_i}N_i^{\perp}$ is $S^2(\frac{1}{2})$, the round sphere of radius $1/2$ in the first case or a ``thin" $S^2(1/2)$ in the second case. Hence $\xt_q(\Sigma_{\overline{p}_i}X^3)\leq \xt_q(S^2(\frac{1}{2}))$ for any $q\geq 2$.

Using the fact that $\xt_3(S^2(\frac{1}{2}))=\pi/3$, it is easily seen that for any triple of distinct points $x_j, x_k, x_l\in S^2(\frac{1}{2})$, we have
\bdm
\textrm{dist}(x_j, x_k) + \textrm{dist}(x_j,x_l) + \textrm{dist}(x_k, x_l)\leq \pi.
\edm
Thus summing over all the triangles formed by the points $\{\bar{p}_i\}_{i=1}^4$ we find that the sum of their angles should be less than or equal to $4\pi$.  Therefore this sum of angles must be exactly $4\pi$.

\end{proof}


\begin{lemma}
\label{L:5_circle_orbits} If there are $5$ isolated circle orbits $\{N_i\}_{i=1}^{5}$ then, for fixed $1\leq i\leq 5$ and points $p_j\in N_j$, $1\leq j \leq 5$, $j\neq i$, the following hold.
\begin{itemize}
	\item[(1)] For each $i$ and each $p_i\in N_i$, we have $G_{p_i}=S^1$ and its slice representation is the Hopf action.

	\item[(2)] For each $j\neq i$, the sets $\mathrm{Dir}_{p_i}(p_j)$, consist of a single $S^1$ Hopf orbit and come in mutually orthogonal pairs, that is, given $i$, for each set of distinct $j,k,l,m$, up to reordering, we can assume that 
	\begin{equation*}
		\alpha_{ijk}=\alpha_{ilm}=\pi/2
	\end{equation*}
	and the remaining angles are all $\pi/4$.
	\end{itemize}
\end{lemma}


\begin{proof} 

For convenience,  let $i=5$. By lemma~\ref{L:angle_bound}, we have 
\bdm
\alpha_{5jk} + \alpha_{5kl} + \alpha_{5lj}=\pi,
\edm
for the 4 points $p_j, p_k, p_l, p_5$, with  $j,k,l\neq 5$.
For $m=j,k,l$, let 
$$
D_m=\textrm{Dir}_{p_5}(p_m)/T^2_{p_5}.
$$
Observe that $D_m\subset \Sigma_{\bar{p}_5}=S^3/T^2_{p_5}$.

We have already seen in the proof of lemma \ref{L:angle_bound} that the sum of the distances between any set of three distinct points in 
$\Sigma_{\bar{p}_i}$ is equal to $\pi$ for any $i\in \{1, 2, 3, 4, 5\}$. From the geometry of the space of directions
this tells us that $D_j, D_k$ and $D_l$ are points in $\Sigma_{\bar{p}_5}$. In the case where $\Sigma_{\bar{p}_5}=S^2(1/2)$ then either two of them are antipodal (at distance $\pi/2$ in $S^2(1/2)$) or all three of them lie on a great circle. Note that in this last case the three points cannot lie in half of the great circle. In the second case, where $\Sigma_{\bar{p}_5}$ is a ``thin" $S^2(1/2)$,  two of the points must be at distance $\pi/2$. Since this is true for any choice of three of the four possible directions, one may conclude that $\Sigma_{\bar{p}_5}$ cannot be smaller than $S^2(1/2)$. In particular, this implies that the isotropy group of each isolated circle orbit is $S^1$ and that the action is the Hopf action. This proves part (1) of the lemma.

Finally,  one can conclude that any four singular points in the space of directions must lie on a great circle and consist of two pairs of antipodal points. This implies that the sets $D_m$ lie on a great circle, consist of a single point each and must pair off as antipodal points, thus proving part (2) of the lemma.

\end{proof}


\begin{lemma}
\label{L:5_circle_orbits_B}  Let $S^1$ act on $\ccc^2$ by scalar multiplication and suppose that $v, w\in 
S^3(1)\subset \ccc^2$. Then 
either $\dist (S^1(v), S^1(w))=\p2$ or there exists a unique $t\in S^1$ such that $\dist(v, tw)=\dist(S^1(v), S^1(w))=\dist(v, S^1(w))$.
\end{lemma}


\begin{proof}

 Let $\pi: S^3(1)\rightarrow S^2(\frac{1}{2})$ be the Hopf map. The Riemannian submersion metric on $S^3(1)/S^1$ is isometric to $S^2(\frac{1}{2})\subset \rrr^3$. The images of the orbits $S^1(v), S^1(w)$ in $S^3(1)/S^1$ are either separated by $\p2$ or they are joined by a unique minimizing geodesic segment $\bar{\gamma}\subset S^3(1)/S^1$ with $\textrm{Length}(\bar{\gamma})=\dist (S^1(v), S^1(w))$. If $w_0\in S^1(w)$ satisfies 
\[
\dist(v, w_0)=\dist (\pi(S^1(v)), \pi(S^1(w)))=\dist (S^1(v), S^1(w)),
\]
 and $\gamma$ is a minimizing geodesic segment from $v$ to $w_0$ in $S^3(1)$, then $\gamma$ projects to the minimizing curve $\bar{\gamma}$. This implies that $\gamma$ is the unique horizontal lift of $\bar{\gamma}$ starting at $v$ and $w_0$ is unique.

\end{proof}


\begin{proof}[Proof of theorem~\ref{T:4-circles}] 

We will assume that there are at least $5$ isolated circle orbits $\{N_i\}_{i=1}^5$ and will obtain a contradiction. For each $1\leq i\leq 5$, let $p_i\in N_i$ and observe that 
for each pair of points $p_i\in N_i$, $p_j\in N_j$, $i\neq j$, $\mathrm{Dir}_{p_i}(p_j)$ is a subset of the unit normal sphere $S^3\subset T_{p_i}N_i^{\perp}$. 
We will now show that when $\alpha_{ijk}=\pi/2$, the set $\textrm{Dir}_{p_j}(p_k)$ cannot be a single $S^1$ orbit
in the unit normal sphere $S^3\subset T_{p_i}N_{i}^{\perp}$, in contradiction with lemma~\ref{L:5_circle_orbits}.

Assume after relabeling points that $\alpha_{123}=\pi/2$.
Let $\gamma_{12}, \gamma_{13}$ be minimal normal geodesics from $p_1$ to $p_2$ and $p_1$ to $p_3$, respectively. By lemma~\ref{L:5_circle_orbits}, 
$$\angle_{p_1}(\gamma'_{12}(0), \gamma'_{13}(0))=\pi/2.$$
 Now since $\alpha_{123} +\alpha_{213} + \alpha_{312}=\pi$, it follows by the discussion of the equality case in the proof of Toponogov's theorem (cf. \cite{CE}) that 
there exists a flat, totally geodesic triangular surface $\bigtriangleup^2\subset M^5$
with geodesic edges $\gamma_{12}, \gamma_{13}$ and $\eta$, where $\gamma_{23}$ is a minimal geodesic from $p_2$ to $p_3$.

Now, if we replace $\gamma_{13}$ with $t\gamma_{13}$, where $t\in S^1$, we obtain another flat, totally geodesic triangular surface $\bigtriangleup^2_t\subset M^5$ with geodesic edges $\gamma_{12}, t\gamma_{13}$ and $\eta_t$, where $\eta_t$ is a minimal geodesic from $p_2$ to $tp_3\in N_3$.
In particular, 
$$\angle_{p_2}(\gamma_2, \eta_t)=\angle_{p_2}(\gamma_2,  \eta)=\alpha_{213}<\p2.$$ Then, by lemma~\ref{L:5_circle_orbits_B}, there is a unique $t_0\in S^1$ such that 
\linebreak $\angle_{p_2}(\gamma_2, \eta_{t_0})=\angle_{p_2}(p_1, p_3)$. It then follows that, for this $t_0$, $\eta_{t_0}=\gamma_{23}$ and thus
$\bigtriangleup^2_t=\bigtriangleup^2$. Hence $t=e$ and we have a contradiction.

\end{proof}


\begin{corollary} 
Let $M^{n+3}$ be a closed, non-negatively curved manifold with an isometric  $T^{n}$ action. Suppose that $M^*=S^3$ and that there are isolated $T^{n-1}$ orbits. Then there are at most four such isolated $T^{n-1}$ orbits. In particular, if $n\geq 7$ then there are none.
\end{corollary}


\begin{proof}
The first result follows directly from the proof of theorem~\ref{T:4-circles}. The second result follows by proposition~\ref{p:Tn-1orbits}.
\end{proof}

\subsection{Possible components with finite isotropy} 
\label{SS:component_finite_iso}


The following lemma, easily generalized from Rong~\cite{R}, allows us to calculate the 
Betti numbers with $\zzz_p$ coefficients of $M^5$.  
\begin{lemma}\label{l:H2a_finite_extension}
\label{p:R}
Suppose $T^2$ acts isometrically on $M^5$, a closed, simply-con\-nected $5$-manifold. If there are exactly 3 isolated circle orbits, then $H_2(M^5)$ has trivial free rank. If there are exactly $4$ isolated circle orbits, then $H_2(M^5)$ has free rank equal to $1$.
\end{lemma}

We will now show that a weighted graph containing a tree with a vertex of degree three, that is,  a weighted graph containing a weighted claw or a weighted tree, may occur only when there are exactly $3$ isolated circle orbits. With this result, we may then conclude that  neither $\Kl\times S^1$ nor $A$ can never occur as the fixed point set of a finite group.


\begin{proposition}
\label{p:nonorientable}
Let $T^2$ act isometrically on $M^5$, a closed, simply-connected, non-negatively curved $5$-manifold.  If $M^*=S^3$ and there exists a non-orientable $3$-manifold $F^3$ fixed by a $\zzz_2$ subgroup, then the projection of $F^3$ in $M^*$  must belong to a weighted claw and there can be no other singular points in $M^*$ corresponding to an isolated circle orbit, besides the external vertices of the claw.
\end{proposition}


\begin{proof}

Let $W$ be the weighted graph corresponding to the set of orbits with non-trivial isotropy in $M^*$. There are two cases we must exclude. The first case is where $W$ contains a weighted claw as a subgraph and a vertex of degree $0$, $1$ or $2$ (see, for example, figure~\ref{F:claw_cycle}). The second case is when $W$ contains a weighted tree as a subgraph (see, for example, figure~\ref{F:tree_cycle}).


\begin{figure}
\centering
\includegraphics[scale=0.7]{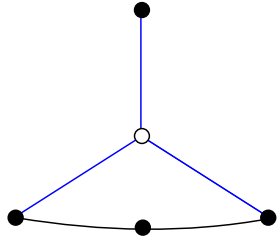}
\caption{Weighted graph containing a claw. The solid vertices correspond to isolated circle orbits. The vertex of degree $3$ corresponds to an exceptional orbit with isotropy $\zzz_2\times\zzz_2$.}
\label{F:claw_cycle}
\end{figure}


\begin{figure}
\centering
\includegraphics[scale=0.7]{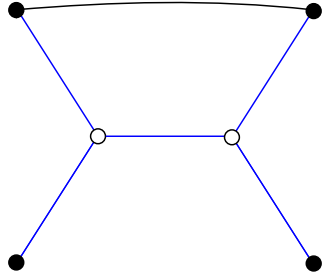}
\caption{Weighted graph containing a tree. The solid vertices correspond to isolated circle orbits. The two vertices of degree $3$ correspond to  exceptional orbits with isotropy $\zzz_2\times\zzz_2$.}
\label{F:tree_cycle}
\end{figure}

We begin with the first case.
Let $\bar{p}_1$ denote the center point in $M^*$ of the weighted claw, 
that is, whose space of directions $\Sigma_{\bar{p}_1}$ is the double right-angled spherical triangle $S^2/(\zzz_2\times \zzz_2)$, and let $\bar{p}_i$, $i=2,3,4$, denote the points in $M^*$ corresponding to the vertices of the weighted claw, each of which corresponds to an orbit with $T^1$ or $T^1\times\zzz_2$ isotropy.
We note that the space of directions for each of these external vertices is either an 
$S^2(1/2)/\zzz_2$, that is a ``thin" $2$-sphere of diameter $\pi/2$ or a possibly thinner $2$-sphere of diameter $\pi/2$. If there is a fourth singular point $\bar{p}_5$ corresponding to an isolated circle orbit in $M^5$, then $\Sigma_{\bar{p}_5}$ is either an $S^2(1/2)$ or a  ``thin" $S^2(1/2)$. 
Since  $S^2(1/2)$ is the ``largest" of these spaces of directions (cf. \ref{l:Lipschitz}), we will assume that $\Sigma_{\bar{p}_5}=S^2(1/2)$.

It is clear that in $\Sigma_{\bar{p}_1}$ the vertices of the spherical triangle correspond to the geodesic directions to the points $\bar{p}_2, \bar{p}_3$ and $\bar{p}_4$, and consequently $\alpha_{123}=\alpha_{124}=\alpha_{134}=\pi/2$. Without loss of generality, we will assume that $\alpha_{135}+\alpha_{145}=\alpha +\beta=\pi/2$, in which case  it follows that $\alpha_{125}=\pi/2$.

Now, by lemma \ref{L:angle_bound}, the tetrahedron $T_{2345}$ is rigid, in the sense that the angles of every face sum to $\pi$ and the angles at every vertex sum to $\pi$. In particular, because of this rigidity and because each of the points $\bar{p}_2, \bar{p}_3$ and $\bar{p}_4$  has space of directions  a thin $S^2(1/2)$, it follows that at every one of the vertices $\bar{p}_2, \bar{p}_3$ and $\bar{p}_4$ 
of $T_{2345}$ there will be an angle of $\pi/2$. 
Further, the maximal configuration for the spaces of directions of the points $\bar{p}_2, \bar{p}_3$ and $\bar{p}_4$ will be where the remaining angles at each vertex in $T_{2345}$ are all $\pi/4$, that is, 
$\alpha_{j1k}=\pi/4$ for for all $j\in \{2, 3, 4\}$ and $k\in \{2, 3, 4, 5\}$, where $j\neq k$, whereas, $\alpha_{51j}$ will be equal to $\pi/2$ for one value of $j\in \{2, 3, 4\}$ and for the remaining values it will be equal to $\pi/4$.
Without loss of generality we may choose specific values for all angles of the form $\alpha_{2jk}, j, k\in \{1, 3, 4, 5\}$. Once these choices are determined, the rigidity of $T_{2345}$ will determine the remaining angles.

It now follows by Toponogov's theorem that the angle sum of any triangle in any tetrahedron formed by these five singular points must be greater than or equal to $\pi$. When we consider the tetrahedron $T_{1345}$, we see that when we substitute all the 
known values for the angles the lower bound on the sum of the angles for any triangle 
forces the following two inequalities:

\bdm
\begin{matrix}
\alpha_{513} +\alpha_{315} +\alpha\geq \pi,\\
\alpha_{514} +\alpha_{415} +\beta\geq \pi.\\
\end{matrix}
\edm

As we saw previously, $\alpha_{415}=\alpha_{315}=\pi/4$ and one of 
$\alpha_{513}$ or $\alpha_{514}$ is equal to $\pi/2$ and the other is equal to $\pi/4$.
In particular this tell us that one of the angles $\alpha$ or $\beta$ is greater than or equal to $\pi/2$ and the other is greater than or equal to $\pi/4$. Since $\alpha +\beta=\pi/2$ this immediately gives us a contradiction and thus this case cannot occur.

For the second case, where the weighted graph contains a weighted tree, we observe that the  addition of the singular point $\bar{p}_6$, corresponding to a $T^2/(\zzz_2\times \zzz_2)$ orbit will produce an analgous contradiction and thus this case cannot occur either.

\end{proof}

We summarize the results of this subsection in the following theorem:


\begin{theorem}\label{t:3dimfiniteisotropy}
Let $T^{2}$ act isometrically on $M^5$, a closed, simply-connected, non-negatively curved $5$-manifold. If $M^*=S^3$,  then the fixed point set components of finite cyclic isotropy (if they exist) are:
\begin{itemize}
\item[(1)] $S^3$, $\lpq$, $S^2\times S^1$, $\rrr P^2\times S^1$ or $S^2\tilde{\times} S^1$ when there are three isolated circle orbits;\\

\item[(2)]  $S^3$, $\lpq$ or $S^2\times S^1$ when there are four isolated circle orbits.
\end{itemize}
\end{theorem}

We recall the following theorem of Bredon~\cite{Br}:



\begin{theorem}
\label{t:cohominequalities}
Suppose that $p$ is a prime and that $G=\zzz_p$ acts on the finite-dimensional space $X$ with $B\subset X$ closed and invariant. Suppose that $G$ acts trivially on ${H}^*(X, B;\zzz)$ and let $F=\Fix(X;\zzz_p)$. Then, for any $k\geq 0$, we have
\bdm
\sum_{i\geq 0} \rk {H}^{k+2i}(F, F\cap B; \zzz_p)\leq \sum_{i\geq 0} \rk {H}^{k+2i}(X, B; \zzz_p).
\edm
\end{theorem}

We observe that any diffeomorphism in $T^2$ is homotopic to the identity, since it is contained in a torus.
Thus we may apply this theorem to the situation at hand to obtain the following corollary:


\begin{corollary}\label{c:orbitconfigs}
Let $T^2$ act isometrically on $M^5$, a closed, simply-connected, non-negatively curved $5$-manifold. If $M^*=S^3$ and the orbit space contains a weighted claw, then $M^5$ is not $S^5$.
\end{corollary}


\begin{proof} 
This follows directly by applying the inequality in theorem~\ref{t:cohominequalities}, observing that if either  $S^2\tilde{\times} S^1$ or $\rrr P^2\times S^1$ is contained in $\Fix(M^5; \zzz_2)$, then  $H_2(M^5)\neq 0$.
\end{proof}



\subsection{Unknottedness of cycles} 
\label{SS:unknot}
We will now analyze the special case where the singular set in the orbit space contains a cycle. Following work of Grove and Wilking \cite{GW,Wi3}, we will show that  this cycle is unknotted in $M^*=S^3$. Recall that the arcs in a cycle correspond to the projection of fixed point sets of finite isotropy. We have the following result.


\begin{theorem} \label{T:unknotted_cycle}
Let $M^5$ be a closed, simply-connected, non-negatively curved $5$-manifold with an isometric $T^2$ action and orbit space  $M^*\simeq S^3$. If the singular set in the orbit space $M^*$ contains a cycle $K^1$, then the following hold:
\begin{itemize}
\item[(1)]  The cycle $K^1$ is the only cycle in the singular set in $M^*$.\\

\item[(2)] If there are four isolated circle orbits, then $K^1$ comprises all of the singular set, i.e., $M^*\setminus K^1$ is smooth.\\ 

\item[(3)] Suppose there are exactly three isolated circle orbits, the cycle $K^1$ contains only two of them, and there are no exceptional orbits of isotropy $\zzz_2\times\zzz_2$. Then the finite isotropy fixing one of the $3$-manifolds corresponding to one of the arcs of the cycle must be $\zzz_2$.\\


 
\item[(4)]  The cycle $K^1$ is unknotted in $M^*$.
 \end{itemize}
\end{theorem}


\begin{proof} 

We will first prove parts (1) and (2) of the theorem.  Our strategy is the following. For any given weighted graph containing a cycle $K^1$ in $M^*$ we will construct a branched double cover $\kappa: B \rightarrow M^*$ with branching set $K^1$ and show that $B$ is a simply-connected Alexandrov space of non-negative curvature (see lemmas~\ref{L:B_nonneg_alex}, \ref{L:B_space_directions} and \ref{L:B_s_connected} below). By lemma~\ref{L:B_space_directions} below and the proof of theorem \ref{T:4-circles}, the cover  $B$   can contain at most four points projecting down to isolated circle orbits in $M^*$. Parts (1) and (2) of the theorem then follow from the fact that  the weighted  graph $W\subset M^*$ corresponding to the set of singular orbits can contain two cycles only if $W$ has four vertices.




To construct the branched cover, first observe that a generator of $H_1(M^*\setminus K^1;\zzz)=\zzz$ is given by a normal circle to $K^1$. Recall that index two subgroups of $H_1(M^*\setminus K^1;\zzz)=\zzz$ are in one-to-one correspondence with two-fold covers of $M^*\setminus K^1$. Hence there is a unique two-fold cover $B'$ of $M^*\setminus K^1$. Let $B_r(K^1)$ be a tubular neighborhood of $K^1$ in $M^*\simeq S^3$. Observe that $B_r(K^1)$ is a solid torus and therefore $H_1(B_r(K^1)\setminus K^1;\zzz)=\zzz^2$.  It follows that $B_r(K^1)\setminus K^1$ also has a two-fold cover. Now let $B=B'\cup K^1$ so that $\kappa: B\rightarrow M^*$ is a two-fold branched cover, with branching set $K^1$. Further, $B$ admits a $\zzz_2$ action, which is isometric with respect to the metric induced by the orbital distance metric from $M^*$,
with fixed point set $K^1$.


\begin{lemma}
\label{L:B_nonneg_alex}
The space  $B$ is a non-negatively curved Alexandrov space. 
\end{lemma}

\begin{proof}[Proof of lemma~\ref{L:B_nonneg_alex}]



Observe that $B$ is locally isometric to $M^*$ outside of the branching set $K^1$. Let $C_2$ be the union of arcs in $K^1$ with isotropy $\zzz_k$, $k\neq 2$. We have two cases: case one, where the cycle $K^1$ contains all the singular points corresponding to isolated circle orbits, and case two, when there are exactly three isolated circle orbits and of the corresponding singular points only two belong to a cycle.

For case one,  proceeding as in the proof of lemma~2.3 in \cite{GW}, one verifies that the set $B \setminus C_2$ is convex in $B$. The conclusion follows after observing that any geodesic triangle in $B$ is the limit of geodesic triangles in $B \setminus C_2$.

For case two, there are only two graphs that correspond to this case: graph 4 of figure~\ref{F:3_vertices_wgraphs} and graph 2 of figure~\ref{F:claw_wgraphs}.
For graph 4 of figure~\ref{F:3_vertices_wgraphs},  we may proceed as in  case one. Here  we verify that the set $B\setminus\{C_2\cup \{p_1, p_2\}\}$ is convex in $B$, where $p_1$ and $p_2$ are the lifts of the point $p$ corresponding to the isolated circle orbit in $M^*$ that does not belong to the cycle $K^1$. For graph 2 of figure~\ref{F:claw_wgraphs}, let $A_1$ be  the arc in $K^1$ with isotropy $\zzz_k$, $k\neq 2$, and let $A_2$ be the lift in $B$ of the arc in the claw with isotropy $\zzz_2$  not contained in $K^1$.  Observe that $A_2$ is a minimal geodesic between the lifts of the isolated circle orbit not contained in the cycle $K^1$. As above, one verifies that the set $B\setminus { (A_1 \cup A_2)}$ is convex in $B$ and the conclusion follows after observing that geodesics triangles in $B$ are limits of geodesic triangles in $B\setminus { (A_1 \cup A_2)}$.

\end{proof}


\begin{lemma}
\label{L:B_space_directions}
Let $\overline{p}\in K^1\subset M^*$ be a point corresponding to an isolated circle orbit in $M^5$ and consider $\overline{p}$ as a point in $B$. Then $\Sigma_{\overline{p}} B$ is smaller than or equal to  $S^2(1/2)$.
\end{lemma}

\begin{proof}[Proof of lemma~\ref{L:B_space_directions}]
There is a two-fold branched cover $\Sigma_{\overline{p}} B \rightarrow \Sigma_{\overline{p}}M^*$. We know that the space of directions $\Sigma_{\overline{p}} M^*$ is a ``thin" $S^2(1/2)$. We will denote this space by $X_{k, l}$.
Observe that $\Sigma_{\overline{p}} M^*=X_{k, l}$ can be written as the join of a circle, $S^1/\zzz_{kl}$, of diameter  $\pi/kl$ and $S^0$, of diameter $\pi/2$, where the former is the normal space of directions to $K^1$ at the point $\overline{p}$ and  the latter corresponds to the tangent space of directions at $\overline{p}$ of $K^1$ in $M^*$. Since $B$ is a two-fold branched cover of $M^*$ with branching locus $K^1$, the corresponding space of normal directions in $B$ will be of twice the diameter as the space of normal directions in $M^*$. In particular, this means that $\Sigma_{\overline{p}} B$ corresponds to
 the orbifold $X_{2k, l}$ or $X_{k, 2l}$.  Since at least one of $k, l$ is greater than $2$, it follows that this orbifold is smaller than or equal to $S^2(1/2)$, as we saw earlier (cf. lemma~\ref{l:Lipschitz}).
\end{proof}


\begin{lemma}
\label{L:B_s_connected} 
The space $B$ is simply-connected.
\end{lemma}

\begin{proof}[Proof of lemma~\ref{L:B_s_connected}]
We will prove this by contradiction, so assume that  $\pi_1(B)$ contains at least two elements. Observe that $\tilde{B}$, the universal cover of $B$,  is a compact Alexandrov space of non-negative curvature. There are at least three singular points $p_i$ in $B$, corresponding to isolated circle orbits in $M^*$. Therefore, in $\tilde{B}$ there will be at least six points $\tilde{p}_k$ covering these points $M^*$.  By lemma~\ref{L:B_space_directions},  the spaces of directions  $\Sigma_{\tilde{p}_k}\tilde{B}$ are smaller than or equal to $S^2(1/2)$. On the other hand, the Extent lemma implies that there can be at most five such points in $\tilde{B}$, yielding a contradiction.
\end{proof}

Now we prove part (3).  Let $\zzz_k$ and $\zzz_l$ be the finite isotropy groups fixing the two $3$-manifolds corresponding to the two arcs of the cycle $K^1$. Without loss of generality, we may assume that $2\leq k <l$. In this case we may take a $k$-fold branched cover of $M^*$ with branching locus $K^1$. It follows from the proof of theorem \ref{T:4-circles} that $k=2$, since otherwise the branched cover would have more than four singular points.

Finally, we prove part (4). Observe that $B$ is a topological $3$-manifold and, by lemma~\ref{L:B_s_connected}, $B$ is simply connected.  Hence, by the resolution of the Poincar\'e conjecture, $B$ must be homeomorphic to $S^3$. Recall that we have an isometric $\zzz_2$ action on $B$ fixing $K^1$. By the equivariant version of the Poincar\'e conjecture \cite{DiL}, it follows that this action is equivalent to a linear action on a standard $S^3(1)$.
Therefore $\zzz_2\subset SO(4)$. Note that $\zzz_2$ is not equivalent to the action of $-\txt{Id}$, since there is a branching locus, that is a unique circle fixed by the $\zzz_2$ action. Thus, without loss of generality,
\bdm
\zzz_2=\{ \begin{pmatrix}
 -1 &&&\\
 & -1 &&\\
 && 1 &\\
 &&& 1 \\
 \end{pmatrix} \cup \txt{Id}\}.
 \edm
 Therefore $K^1\subset M^*=B/\zzz_2$ is not knotted. This concludes the proof of theorem~\ref{T:unknotted_cycle}.
\end{proof}



\section{Non-negatively curved $5$-manifolds with almost maximal symmetry rank}
\label{s:dim5results}
We can now classify closed, simply-connected, non-negatively curved $5$-man\-ifolds with an isometric $T^2$ action, corresponding to the almost maximal symmetry rank case in dimension $5$.  We summarize our results in the following theorem.


\begin{theorem} \label{T:dim5} Let $M^5$ be a closed, simply-connected $5$-manifold with non-negative curvature admitting an isometric $T^2$ action. Then $M^5$ is diffeomorphic to $S^5$, $\linebreak S^3\times S^2$, $S^3\tilde{\times}S^2$ or $SU(3)/SO(3)$.
\end{theorem}

By lemma~\ref{l:nofreeoralmostfree}, the $T^2$ action is neither free nor almost free. In particular, this tells us that there is always some circle subgroup with non-empty fixed point set. 
We then have two cases to consider: case A, where some circle subgroup acts  fixed point homogeneously and  therefore $M^*$ is $D^3$ or $S^2\times I$, and case B, where no circle subgroup acts fixed point homogeneously and hence $M^*=S^3$. Throughout, our main goal will be to compute $H_2(M^5)$. The conclusions of  theorem~\ref{T:dim5} will then follow by the Barden-Smale classification of simply-connected $5$-manifolds \cite{Ba,Sm}.
We remark that  it is only in case B, where $M^*=S^3$ and we have finite isotropy, that  the Wu manifold, $SU(3)/SO(3)$, may occur. 
Observe also that if one circle subgroup acts freely, then $M^5$ fibers over one of the four manifolds  $\ccc P^2, S^2\times S^2$ or $\ccc P^2\#\pm \ccc P^2$ (cf.  theorem~\ref{T:sy1}). The corresponding total space is diffeomorphic to $S^5$, $S^3\times S^2$ or  $S^3\tilde{\times}S^2$ (cf. \cite{DL}).  It follows that in the case where we obtain the Wu manifold, no circle subgroup may act freely.


\subsection{Case A: $\d M^*\neq\emptyset$}\label{s:5b}

Let $M$ be a non-negatively curved manifold with a fixed point homogeneous $T^1$ action. By definition, this means that there is a component $F$ of $\Fix(M;T^1)$ of codimension $2$.
The following proposition restricts the fundamental group of $F$ depending on the dimension of the set at maximal 
distance to $F^*$ in the orbit space $X=M/T^1$ of the action. We let $\pi:M\rightarrow X$ be the orbit projection map of the fixed point homogeneous circle action.


\begin{proposition} 
\label{P:1rank}
Let $M^n$ be a closed, simply-connected, non-negatively curved manifold of dimension $n\geq 4$ with an isometric $T^1$ action and suppose that  \linebreak 
$\Fix(M^n; T^1)$ contains an $(n-2)$-dimensional component $F^{n-2}$.  Let $C^k$ be the set at maximal distance from $F^{n-2}$ in the orbit space $X^{n-1}=M^n/T^1$.

\begin{itemize}

\item[(1)] If $C^k$ has dimension $k=n-2$,  then $C^k$ is fixed by the $T^1$ action, is isometric to $F^{n-2}$ and $F^{n-2}\cong C^{n-2}$ is simply-connected.\\

\item[(2)]If $C^k$ has dimension $k \leq n-4$, then $F^{n-2}$ is simply-connected.\\

\end{itemize}
\end{proposition}

\begin{proof} 
 
 First we prove (1).  If we suppose that $C$ is not fixed by the $S^1$ action, then $B=\pi^{-1}(C)$ has dimension $n-1$. We may decompose $M$ as a union of neighborhoods of $N^{n-2}$ and  $B$, which we will denote $V$ and $U$, respectively.
Their common boundary is a circle bundle over $N^{n-2}$ which we denote by $\d V$. 
Observe that both $V$ and $\d V$ are orientable, but that $U$ is not (this follows from the Mayer-Vietoris sequence of the triple $(M, V, U)$). Since $M^n$ is simply-connected
it follows by duality that $H_{n-1}(M)=0$ and that the torsion subgroup of $H_{n-2}(M)$ is trivial. Further, since $\d V$ is a compact, orientable manifold of dimension $n-1$, the torsion subgroup of $H_{n-2}(\d V)$ is also trivial.
Likewise, since $V$ deformation retracts onto $N^{n-2}$, $H_{n-2}(V)=\zzz$. Since $B$ is non-orientable, the torsion subgroup of $H_{n-2}(U)$ is equal to $\zzz_2$. If we write down the first few elements of the Mayer-Vietoris sequence of the triple $(M, V, U)$ we have:

\bdm
\begin{matrix}
0\ra H_n(M)\ra H_{n-1}(\d V)\ra H_{n-1}(U)\oplus H_{n-1}(V)\ra H_{n-1}(M)\\
\ra H_{n-2}(\d V)\ra H_{n-2}(U)\oplus H_{n-2}(V)\ra H_{n-2}(M).\\
\end{matrix}
\edm
 Substituting known values we obtain:
 
 \bdm
\begin{matrix}
0\ra \zzz\ra \zzz \ra 0 \oplus 0 \ra 0\\
\ra \txt{Fr}(H_{n-2}(\d V))\ra (\txt{Fr}(H_{n-2}(U))\oplus \zzz_2) \oplus \zzz \ra \txt{Fr}(H_{n-2}(M)).\\
\end{matrix}
\edm
The sequence is not exact and thus this case cannot occur. This in turn implies that the inverse image of $C^{n-2}$ in $M$  must be exactly $C^{n-2}$ and thus $C^{n-2}$ is a component of $\Fix(M; T^1)$. It then follows from the Double Soul Theorem \ref{T:sy2} that $M$ is an $S^2$ bundle over $F^{n-2}$ and hence $N^{n-2}$ must be simply-connected.
 
To prove (2), let $\gamma$ be a loop in $F^{n-2}\subset M^n$. Since $M^n$ is simply-connected, $\gamma$ bounds a $2$-disk $D^2$. Let $B^{k'} = \pi^{-1}(C^k)$ and observe that $k'\leq n-3$. By transversality, we can perturb $D^2$ so as to lie in the complement of $D(B^{k'})$, a tubular neighborhood of $B^{k'}$, while keeping $\gamma=\partial D^2$ in $F^{n-2}$. The conclusion follows after observing that $D^2$ is now contained in $D(F^{n-2})$, which deformation retracts onto $F^{n-2}$. 
\end{proof}


\begin{remark}
The assertions in Proposition \ref{P:1rank} hold trivially in dimension $2$, since in this case the fixed point set components of codimension $2$ are isolated points. 
In dimension $3$, however, proposition~\ref{P:1rank}  fails, since a $1$-dimensional fixed point set component, being compact, must be a circle and thus has infinite fundamental group.
\end{remark}

We will now classify non-negatively curved simply connected $5$-manifolds $M^5$ with an isometric $T^2$ action such that some circle subgroup $T^1$ acts fixed point homogeneously. We remark that simply connected $5$-manifolds with non-negative curvature and a fixed point homogeneous isometric circle action were classified in \cite{GGSp} and that the resulting classification is the same in both cases. We further remark that the methods of proof are distinct.


 In the following two propositions we will characterize $F^3$ in the case where $B=\pi^{-1}(C)$ is of dimension three. We must consider two cases: case one, where $B^3$ is not fixed by any circle action and case two, where $B^3$ is fixed by a distinct circle subgroup of $T^2$. We begin with case one.


\begin{proposition}
\label{p:noSeiferts}
 Let $M^5$ be a closed, simply-connected, non-negatively curved $5$-manifsold admitting an isometric $T^2$ action. Suppose that $F^3$ is the unique $3$-dimensional fixed point set component of $\Fix(M^5; T^1)$ for some $T^1\subset T^2$ and that $B=\pi^{-1}(C)$ is of dimension $3$. Further suppose that $B$ is not the fixed point set of any circle subgroup of $T^2$.
 Then $F^3$ is diffeomorphic to one of $S^3$, $\lpq$, $S^2\times S^1$ or $\rrr P^3\#\rrr P^3$.
\end{proposition}


\begin{proof}
Observe that $F^3$ is orientable, but may not be simply-connected. Since it admits a circle action it must be one of the $3$-manifolds listed in  theorem~\ref{T:n3.1}.  The circle action 
can be free, almost free or admit fixed points. Observe that the action cannot be $\zzz_k$-ineffective. 

Recall from the proof of theorem~\ref{T:n3.1} that the $3$-manifolds with  spherical geometry,  $S^2(2,3,3)$, 
$S^2(2,3,4)$, $S^2(2,3,5)$, $S^2(2,2,n)$, $n\geq 2$,  and the ones with Euclidean geometry, $S^2(2, 2, 2, 2)$, $S^2(3, 3, 3)$, $S^2(2, 4, 4)$ or $S^2(2, 3, 6)$
may only occur when the circle action is almost free and that $T^3$ may occur for any type of circle action. We will first show that if the action is almost free, then the only possibilities for $F^3$ that can occur are the ones listed in the proposition, as well as $T^3$. Finally, we will show that $T^3$ cannot occur for any type of action. 

Thus we assume $F^3$ to be one of the $3$-manifolds with spherical or Euclidean geometry listed above.
Since the circle action is almost free, the image of $F^3$ in the quotient space, $F^3/T^2\subset M^*$, is  a connected subspace 
of the boundary $\d M^*=S^2$.  By the Soul theorem for Alexandrov spaces, the distance function from 
 $\d M^*=S^2$  is concave and therefore any singular points, that is, points with space of directions of 
diameter $\leq \pi/2$, in $M^*$ must be at maximal 
distance from $F^3/T^2\subset \d M^*$. In particular, these singular points correspond to orbits of the $T^2$ action in $M^5$
that are singular or exceptional and in the latter case they must be of the form $T^2/(\zzz_2\times \zzz_2)$ (cf. the proof of proposition \ref{p:3dimfixed}).

Since we assume that the circle action on $F^3$ is almost free,  
there will be fixed circles of finite isotropy. By our previous analysis of the isotropy representations, 
these fixed circles must belong to a $3$-dimensional submanifold of 
finite isotropy, say $\zzz_k$, intersecting $F^3$ in the same circle. We will denote this $3$-manifold by $N^3_k$. 

Observe that $N^3_k$ admits a cohomogeneity one $T^2$ action and 
it must have at least one singular orbit, which is precisely the circle that belongs to $F^3\cap N^3_k$.
Now we note that if $F^3$ were to contain at least two circles fixed by the same $\zzz_k$ subgroup, it might be possible 
for these two circles to correspond to the two distinct singular orbits of a $N^3_k$. 
If this were the case, then either there must be a unique principal orbit or there is an  $(T^2/\zzz_k)\times I$ at maximal distance from $F^3$.  If there is a unique principal orbit of $N^3_k$ at maximal distance from $F^3$ where $N^3_k$ ``turns around" to head back to $F^3$,
we then have a point where there are four distinct directions fixed by the corresponding isotropy subgroup and therefore all of $M^5$ would be fixed by this isotropy, making 
the action $\zzz_k$-ineffective, contrary to hypothesis. Likewise, if we have a $(T^2/\zzz_k)\times I$ at maximal distance, then we will have two points where this happens, again 
giving us a contradiction.

 Since $B^3$ is convex, it is a non-negatively 
curved Alexandrov space admitting a $T^2$ action. By hypothesis,  this action must be of cohomogeneity one 
 and therefore by \cite{GGS2} $B^3$ is an invariant submanifold of $M^5$ (cf. also \cite{GGSp}). 
 Since $B^3$ may have at most two non-principal orbits it follows easily that none of the $3$-manifolds
   with  spherical geometry  $S^2(2,3,3)$, 
$S^2(2,3,4)$, $S^2(2,3,5)$, $S^2(2,2,n)$, $n\geq 2$,  nor the ones with Euclidean geometry $S^2(2, 2, 2, 2)$, $S^2(3, 3, 3)$, $S^2(2, 4, 4)$ or $S^2(2, 3, 6)$
may occur. 

Now we have two cases to consider: case 1, where $N^3_k$ is orientable and case 2, where $N^3_k$ is non-orientable. Recall that $N^3_k$ is only non-orientable when $k=2$. In the first case, it follows from the cohomogeneity one decomposition 
that $N^3_k$
has two singular orbits (cf. \cite{M}, \cite{N}) and in the second case, it will have one singular and one exceptional orbit. One of these singular orbits will belong to 
$F^3$ and since singular orbits not belonging to $F^3$ 
must be contained 
in the set at maximal distance from $F^3$, the other orbit is contained in $B^3$. 

For the manifolds $S^2(2, 3, 3)$, $S^2(2, 3, 4)$, $S^2(2, 3, 5)$,  $S^2(3, 3, 3)$, $S^2(2, 4, 4)$ and $S^2(2, 3, 6)$ it follows  that we must have at least two singular orbits in $B^3$ corresponding 
to the intersection of the corresponding $N^3_k$ , $k\neq 2$, with $B^3$ and another singular or exceptional orbit corresponding to the intersection of $N^3_k$ with $B^3$ for the third isotropy group $\zzz_k$.  This yields a contradiction, since $B^3$ can have at most two non-principal orbits.
For the manifolds $S^2(2, 2, n)$ and $S^2(2, 2, 2, 2)$  if at least two of the isotropy subgroups correspond to orientable $N^3_k$, then we obtain a contradiction in exactly the same manner.

We now consider the case where at least two of the isotropy groups in $S^2(2, 2, n)$ and at least three in  $S^2(2, 2, 2, 2)$ correspond to a non-orientable $N^3_k$.  
Observe that all the $\zzz_2$ isotropy subgroups of these manifolds belong 
to a unique circle subgroup of $T^2$ and are therefore all the same $\zzz_2$ subgroup.  As we saw in our analysis of the isotropy groups in 
proposition \ref{p:3dimfixed}, at an exceptional orbit the three $3$-manifolds that 
must share this exceptional orbit must also have distinct $\zzz_2$ isotropy. If there is an exceptional orbit, then for every 
circle fixed by a $\zzz_2$ in $F^3$ we would have two more $3$-manifolds fixed by distinct $\zzz_2$ 
subgroups and whose other singular orbit must be contained in $F^3$, since the exceptional orbit is already contained in $B^3$. 
This, however, is impossible and so this case cannot occur.

Thus when the action is almost free only $S^3$, $\lpq$, $S^2\times S^1$, $T^3$ or $\rrr P^3\# \rrr P^3$ may occur.

 Finally, $T^3$ is ruled out, since $T^3$ has quotient $T^2$ modulo a circle action, and this would imply that $T^2\subset \d M^*$,   where  $M^*=M^5/T^2$. This, however, is a contradiction, since the connected components of $\d M^*$ are homeomorphic to $S^2$.
\end{proof}

We now consider case two.


\begin{lemma}
\label{l:B3=F3}
Let $T^2$ act isometrically on $M^5$, a closed, non-negatively curved, simply-connected $5$-manifold. Suppose that $\d M^*\neq \emptyset$,  
and both $F^3$ and $B^3$ are fixed by different circle subgroups. If $F^3$ is one of the spherical manifolds, $S^2(2,3,3)$, 
$S^2(2,3,4)$, $S^2(2,3,5)$, $S^2(2,2,n)$, $n\geq 2$,  or  one of the flat manifolds, $S^2(2, 2, 2, 2)$, $S^2(3, 3, 3)$, $S^2(2, 4, 4)$ or $S^2(2, 3, 6)$, then
$B^3$ must be diffeomorphic to $F^3$.
\end{lemma}


\begin{proof}

As we saw in proposition \ref{p:noSeiferts}, if $F^3$ is one of the manifolds listed in the statement of the lemma, then there will be a three-dimensional manifold $N^3_k$
fixed by a $\zzz_k$ subgroup and  intersecting $F^3$ in a circle. Each $N^3_k$ will also intersect $B^3$ 
in a circle and thus both $F^3$ and $B^3$ are diffeomorphic to one another.

\end{proof}


\begin{proposition}
\label{p:B3fixed}
Let $T^2$ act isometrically on $M^5$, a closed, non-negatively curved, simply-connected $5$-manifold. Suppose that $\d M^*\neq \emptyset$,  
and both $F^3$ and $B^3$ are fixed by different circle subgroups. Then both $F^3$ and $B^3$ are diffeomorphic to one of $S^3$, $\lpq$, $S^2\times S^1$ or $\rrr P^3\#\rrr P^3$.
\end{proposition}


\begin{proof}
Observe first that, if $F^3$ is one of the spherical manifolds, $S^2(2,3,3)$, \linebreak 
$S^2(2,3,4)$, $S^2(2,3,5)$, $S^2(2,2,n)$, $n\geq 2$, or  one of the flat manifolds, \linebreak $S^2(2, 2, 2, 2)$, $S^2(3, 3, 3)$, $S^2(2, 4, 4)$ or $S^2(2, 3, 6)$, by lemma \ref{l:B3=F3}, $B^3$ is diffeomorphic to $F^3$.
Now, we consider the isolated singular points in  the orbit space $M^*=M^5/T^2$ corresponding to the circles fixed by a finite cyclic group in $B^3$ and $F^3$.
For each of these points, the corresponding space of directions will be the join of an $S^1/\zzz_k$ and a single point at distance $\pi/2$ from this circle. That is, it will be a positively curved $2$-disc of diameter $\pi/2$. 
There will be six or eight such isolated singular points in $M^*$ and it is easily seen using the Extent Lemma \ref{L:xt}  that this gives us a contradiction 
to the non-negative sectional curvature condition. Thus none of these possibilities is allowed.

As in the proof of proposition \ref{p:noSeiferts}, we see that  $T^3$ cannot occur and thus the only manifolds 
left  are $S^3$, $\lpq$, $S^2\times S^1$ and $\rrr P^3\# \rrr P^3$.

\end{proof}
We are now in a position to  prove the following theorem:


\begin{theorem}\label{t:FPH} Let $T^2$ act isometrically on $M^5$, a closed, non-negatively curved, simply-connected manifold. If $\d M^*\neq \emptyset$,  then $M^5$ is diffeomorphic to $S^5$, $S^3\times S^2$ or $S^3\tilde{\times} S^2$.
\end{theorem}


\begin{proof}


By hypothesis, we have a circle $T^1\subset T^2$ acting fixed point homogeneously on $M^5$ with orbit space $X^4$.  Let $F^3$ be a component of $\Fix(M^5;T^1)$.
We first observe that if there are two three-dimensional components of  $\Fix(M^5;T^1)$, then by the Double Soul Theorem \ref{T:sy2}
there are exactly two such components, they are isometric and $M^5$ is diffeomorphic to an $S^2$-bundle over $F^3$. Since $M^5$ is simply-connected, $F^3$ must also be simply-connected and therefore $F^3$ must be $S^3$ and, as there is only one $S^2$ bundle over $S^3$ (cf. \cite{St}),
$M^5$ is diffeomorphic to $S^2\times S^3$.

 We now consider the case where there is a unique three-dimensional component, $F^3$ of $\Fix (M^5; T^1)$, and let $C^k$ be the set at maximal distance from $F^3$ in the orbit space $X^4$. We have $0\leq k \leq 3$. If $k=3$, then $C$ is a component of $\Fix(M; S^1)$ and as we saw above, $M^5$ is diffeomorphic to $S^3\times S^2$.

 We first consider the case where $k=0$ or $1$. In this case $F^3$ is simply-connected by proposition~\ref{P:1rank} and is therefore $S^3$. If $C^k$ has dimension $0$, then it is  the soul of $X^4$ and, as in \cite{GS}, $M^5$ is diffeomorphic to $S^5$. If $C^k$ is 1-dimensional, it must be an interval or a circle. Since $F^3\cong S^3$ is homeomorphic to the boundary of a neighborhood of $C^1$, we see that $C^1$ is an interval. By Kleiner's Isotropy Lemma~\ref{l:Kleiner}, the only points in $C^1$ that may have finite non-trivial isotropy 
are the endpoints. In the presence of points with finite isotropy, $F^3$ must be a lens space or the connected sum of two lens spaces, neither of which is possible. Hence all the points in $C^1$ are regular, $X^4$ is a manifold with boundary $F^3=S^3$, the soul of $X^4$ is a regular point and again we conclude that $M^5$ is $S^5$.

Suppose now that $C^k$ is $2$-dimensional.  By proposition~\ref{p:noSeiferts} and proposition \ref{p:B3fixed}, $F^3$ can only be one of $S^3$, $\lpq$, $S^2\times S^1$ or $\rrr P^3\#\rrr P^3$.
Recall that $B^3=\pi^{-1}(C^2)\subset M^5$ is the lift of $C^2$ under the orbit projection map $\pi:M^5\rightarrow X^4$. By proposition \ref{p:B3fixed}, we know that 
if $B^3$ is fixed by some distinct circle subgroup, then it too can only be one of $S^3$, $\lpq$, $S^2\times S^1$ or $\rrr P^3\#\rrr P^3$.
Thus it remains to characterize $B^3$ when it  is not fixed by any circle subgroup. To do so we use the fact that $B^3$ admits a cohomogeneity one $T^2$ action. Once we do this, we will have two cases remaining: case one, where $F^3$ admits an almost free circle action and there are then $3$-manifolds in $M^5$ fixed by finite isotropy and case two, where there is no finite isotropy. We will see that in both cases we may decompose $M^5$ as the union of two disc bundles over $F^3$ and $B^3$, respectively, and we may then directly calculate the homology groups of $M^5$ using the Mayer-Vietoris sequence to classify $M^5$ using the results of Smale and Barden
(cf. \cite{Sm}, \cite{Ba}).


Suppose then that $T^2$ acts on $B^3$ by cohomogeneity one.
In order to understand what the possibilities for $B^3$ are we recall that 
$C^2/S^1$ must be an interval and therefore $C^2$ can only be one of  $S^2$, $D^2$, $\rrr P^2$ or a cylinder $S^1\times I$, since it is non-negatively curved.

We first analyze the possible isotropy configurations on $C^2/S^1 \simeq [0,1]=I$. 
By our previous analysis of the isotropy representations (cf. proposition~\ref{p:3dimfixed}), it follows that there can be no exceptional orbits of the circle action in $C^2$. The action can, however, be $\zzz_k$-ineffective, in which case $B^3$ must also be fixed by $\zzz_k$ and the  $T^2$ cohomogeneity one action is $\zzz_k$-ineffective. In this particular case it follows that $B^3$ can only be $S^2\times S^1$, $S^3$ or $L_{p,q}$. We will therefore assume that the circle action on $C^2$ is effective. 
Then the endpoints of $I$ correspond to either  two singular orbits,  one singular orbit and one principal orbit, or two principal orbits. In the first case, $B^3$ corresponds to $S^3$ or $L_{p,q}$. In the second case, $B^3$ corresponds to $S^2\times S^1$. In the last case, $C^2$ must be a cylinder $S^1\times I$. In this case the soul of $X^4$ is a circle and $F^3$ must be $S^2\times S^1$. This case also corresponds to the case in which $B^3$ is also $S^2\times S^1$.

We summarize our results in the following proposition.


\begin{proposition}
Let $T^2$ act isometrically on $M^5$, a closed, non-negatively curved, simply-connected $5$-manifold. Let $\d M^*\neq \emptyset$
and suppose that $F^3$ is the unique $3$-dimensional component of $\Fix(M^5; S^1)$ for some $S^1\subset T^2$ and that $B=\pi^{-1}(C)$ is of dimension three.
\begin{itemize}

\item[(1)] If $B^3$ is fixed by a distinct circle subgroup of $T^2$, then $F^3$ and $B^3$ are diffeomorphic to one of $S^3$, $\lpq$, $S^2\times S^1$ or $\rrr P^3\#\rrr P^3$.\\

\item[(2)] If $B^3$ is not fixed by any circle subgroup of $T^2$, then $F^3$ is diffeomorphic to one  of $S^3$, $\lpq$, $S^2\times S^1$ or $\rrr P^3\#\rrr P^3$ and 
$B^3$ is diffeomorphic to one of  of $S^3$, $\lpq$ or $S^2\times S^1$.
\end{itemize}
\end{proposition}

We may now complete the proof of the theorem. 
Note first that if $T^2$ acts without finite isotropy, then it is clear that we may decompose $M^5$ as a union of disc bundles over $F^3$ and $B^3$, respectively.
Since $M^5$ is simply-connected, it is not possible for either $F^3$ or $B^3$ to be $\rrr P^3\#\rrr P^3$. All other combinations of $F^3$ and $B^3$ are possible.  In all cases we see from the Mayer-Vietoris sequence of the triple $(D(F^3),D(B^3),\partial D(F^3))$ that $H_2(M^5)=\zzz$ and thus $M^5$ is diffeomorphic to one of the $S^3$ bundles over $S^2$ by results of 
Smale \cite {Sm} and Barden \cite{Ba}.
 
 Finally, if there is finite isotropy, in the orbit space, $M^*=M^5/T^2$, we see that there may be at most two intervals corresponding to three-dimensional manifolds fixed by finite cyclic groups joining at most two distinct points in the image of $F^3$ in $M^*$ to at most two distinct points of   the image of $B^3$ in $M^*$. We further observe that the angle that the interval corresponding to the image of the three-dimensional submanifold of finite isotropy forms with these points  is $\pi/2$. Therefore at the points of the image of $F^3$ in $M^*$, the normal space to the corresponding points of $F^3$ will be the tangent space to the three-dimensional submanifold fixed by finite isotropy. The same is true for the corresponding points in $B^3$. If we then consider the unit normal vector field to $F^3$ in $M^5$, this then allows us  to 
 isotope the boundary of a tubular neighborhood around $F^3$ onto the boundary of a tubular neighborhood of $B^3$
   and we may decompose $M^5$ as a union of two disc bundles over $F^3$ and $B^3$, where, as above, each may be one of $S^3$, $\lpq$ or $S^2\times S^1$ and neither may be $\rrr P^3\# \rrr P^3$. As in the case without isotropy,   we see that $M^5$ is diffeomorphic to one of the $S^3$ bundles over $S^2$ by results of 
Smale \cite {Sm} and Barden \cite{Ba}.

This concludes the proof of theorem~\ref{t:FPH}
\end{proof}


\subsection{\bf Case B: ${\bf M^*=S^3}$}
\label{s:5_dim_1}
We consider now the case where  no circle acts fixed point homogeneously, that is, $M^*=S^3$, and there are either three or four isolated circle orbits. We will prove the following theorem:

\begin{theorem}
\label{T:only-fixed-circles}  Let $M^5$ be a closed, simply-connected, non-negatively curved $5$-manifold admitting an isometric $T^2$ action. If  $M^*=S^3$, then $M^5$ is diffeomorphic to $S^5$, $SU(3)/SO(3)$, $S^3\times S^2$ or $S^3\tilde{\times}S^2$.
\end{theorem}


By the proof of proposition~\ref{p:Tn-1orbits}, in the case where there is no finite isotropy, we have the following result:


\begin{proposition}\label{P:M5-T2-CaseA}  Let $M^5$ be a closed, simply-connected, non-negatively curved $5$-manifold admitting an isometric $T^2$ action.
Suppose $M^*=S^3$ and that there is no finite isotropy.
\begin{itemize}
\item[(1)] If there are exactly three isolated circle orbits, then $M^5$ is diffeomorphic to $S^5$.\\
\item[(2)] If there are four isolated circle orbits, then $M^5$ is diffeomorphic to $S^3\times S^2$ or  $S^3\tilde{\times}S^2$.\\
\end{itemize}
\end{proposition}

We now consider the case where the $T^2$ action admits finite isotropy groups. We will devote the rest of this section to the proof of the following proposition:


\begin{proposition}\label{p:finiteisotropy}  Let $M^5$ be a closed, simply-connected, non-negatively curved $5$-manifold with an isometric $T^2$ action. Suppose that $M^*=S^3$ and there is finite isotropy.
\begin{itemize}
\item[(1)]  If there are exactly three isolated circle orbits, then $M^5$ is diffeomorphic to $S^5$ or the Wu manifold, $SU(3)/SO(3)$.\\
\item[(2)]  If there are four isolated circle orbits, then $M^5$ is diffeomorphic to $S^3\times S^2$ or  $S^3\tilde{\times}S^2$.\\
\end{itemize}
\end{proposition}



We first note that combining  proposition~\ref{p:singsets}, theorem \ref{T:4-circles} and theorem~\ref{T:unknotted_cycle},
we may now make a complete list of all the graphs that may occur in the case where we have three isolated circle orbits and in the case where we have four isolated circle orbits.
We list these graphs in  figures~\ref{F:3_vertices_wgraphs}, \ref{F:claw_wgraphs} and \ref{F:4_vertices_wgraphs}.

Recall now that, by theorem~\ref{T:unknotted_cycle}, if the weighted graph contains a cycle, then this cycle must be unknotted in $M^*=S^3$.  We will now show in all cases where we have a cycle that we may decompose the 
manifold as a union of disc bundles, where at least one of the disc bundles is over one arc of the cycle.

We have the following proposition:


\begin{proposition}
\label{p:bundledecomp}
 Let $M^5$ be a closed, simply-connected, non-negatively curved $5$-manifold with an isometric $T^2$ action. Suppose that $M^*=S^3$ and there is finite isotropy.

Suppose, in addition,  that the weighted graph, $W$, corresponding to the singular set of the action contains a cycle $K^1$, corresponding to graphs (3) and (4) in figure~\ref{F:3_vertices_wgraphs}, graph (2) in figure~\ref{F:claw_wgraphs}, and graph (5) in figure~\ref{F:4_vertices_wgraphs}. 
Then the following are true.
\begin{itemize}
\item[(1)] 
If $W$ is graph (3) in figure~\ref{F:3_vertices_wgraphs}, then $M^5$ decomposes as the union of a disc bundle over a $3$-dimensional submanifold $N^3\subset M^5$ fixed by finite cyclic isotropy, corresponding to the pre-image of an arc in $K^1$, and a disc bundle over the remaining circle orbit not contained in $N^3$.
\\

\item[(2)] If $W$ is graph (4) in figure~\ref{F:3_vertices_wgraphs}, then $M^5$ decomposes as a union of disc bundles over two $3$-dimensional submanifolds. One of these $3$-manifolds corresponds to the fixed point set of $\zzz_k$ isotropy, $k\neq 2$, and the other corresponds to the pre-image of the arc between the remaining isolated circle orbit and an exceptional orbit $T^2/\zzz_2$ which projects to an interior point of the arc of $\zzz_2$ isotropy. 
\\

\item[(3)] If $W$ is graph (2) in figure \ref{F:claw_wgraphs}, then $M^5$ decomposes as the union of disc bundles over two $3$-dimensional submanifolds. One of these $3$-manifolds corresponds to the pre-image of the arc with $\zzz_k$ isotropy, $k\neq 2$, and the other to the pre-image of the arc with $\zzz_2$ isotropy containing the remaining isolated circle orbit.
\\

\item[(4)] If $W$ is graph (5) in figure~\ref{F:4_vertices_wgraphs}, then $M^5$ decomposes as the union of two disc bundles over two disjoint $3$-dimensional submanifolds fixed by finite isotropy (although not necessarily the same group). 

\end{itemize}
\end{proposition}


\begin{figure}
\psfrag{A}{(1)}
\psfrag{B}{(2)}
\psfrag{C}{(3)}
\psfrag{D}{(4)}
\centering
\includegraphics[scale=0.5]{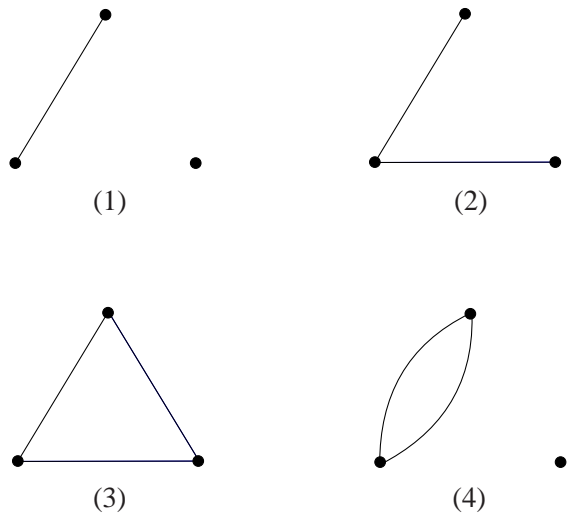}
\caption{Possible weighted graphs when there are exactly three isolated circle orbits and  only finite cyclic isotropy.}
\label{F:3_vertices_wgraphs}
\end{figure}


\begin{figure}
\psfrag{A}{(1)}
\psfrag{B}{(2)}
\centering
\includegraphics[scale=0.5]{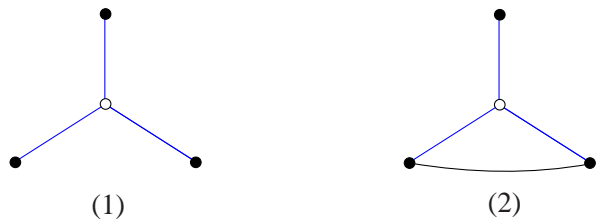}
\caption{Possible weighted graphs when there are exactly three isolated circle orbits and  an isolated exceptional orbit.}
\label{F:claw_wgraphs}
\end{figure}


\begin{figure}
\psfrag{A}{(1)}
\psfrag{B}{(2)}
\psfrag{C}{(3)}
\psfrag{D}{(4)}
\psfrag{E}{(5)}
\centering
\includegraphics[scale=0.7]{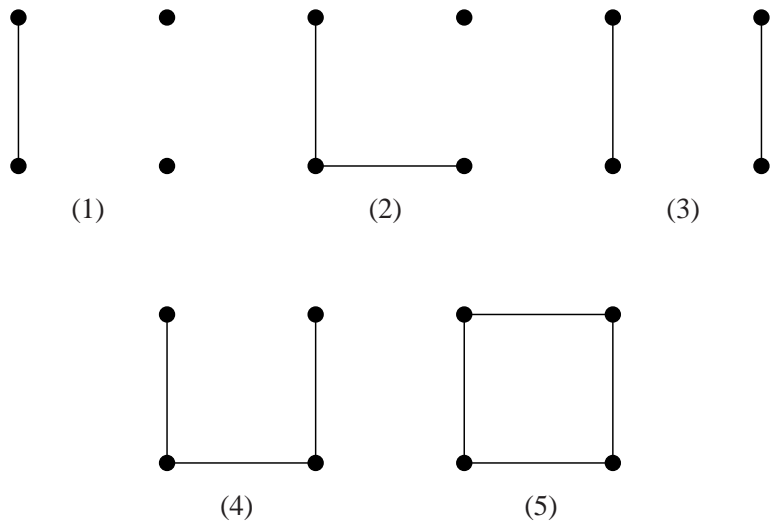}
\caption{Possible weighted graphs when there are exactly four isolated circle orbits and  finite cyclic isotropy.}
\label{F:4_vertices_wgraphs}
\end{figure}


\begin{proof}

We will first prove parts (1) and (4) corresponding, respectively, to graph (3) in figure~\ref{F:3_vertices_wgraphs} and graph (5) in figure~\ref{F:4_vertices_wgraphs}. In both cases the weighted graph is a cycle $K^1$. 

Fix an arc $A_1^*$ in $K^1$ corresponding to a fixed point set component of finite cyclic isotropy $\mathbb{Z}_k$. Note that whether we have three or four isolated circle orbits, the corresponding edges of the weighted cycle $K^1$ in $M^*$ form the angle $\pi/2$. Thus, at isolated circle orbits corresponding to the endpoints of the arc $A^*_1$, the normal space to the $3$-dimensional submanifold $N^3_{\mathbb{Z}_k}=\pi^{-1}(A_1^*)$, fixed by $\mathbb{Z}_k$, will be the tangent space to the $3$-dimensional submanifold fixed by finite cyclic isotropy, corresponding to the lift of arcs adjacent to $A^*_1$. In graph (3) in figure~\ref{F:3_vertices_wgraphs} the cycle $K^1$ contains three edges, and there are three $3$-dimensional submanifolds fixed by finite isotropy, each one corresponding to the lift of an arc in $K^1$. Consider $N^3_{\zzz_k}=\pi^{-1}(A_1^*)$, and a unit normal vector field to it. Since the normal vector field is tangent to the other two $3$-dimensional submanifolds fixed by finite isotropy, we may isotope the boundary of a tubular neighborhood around $N^3_{\zzz_k}$ along this vector field to the boundary of a tubular neighborhood around the remaining isolated circle orbit $C^1$. Therefore we may decompose $M^5$ as a union of disc bundles over $C^1$ and over $N^3_{\zzz_k}$. 

The same argument in the preceding paragraph works for  graph (5) in figure~\ref{F:4_vertices_wgraphs}, corresponding to the case of four isolated circle orbits. Here we will  isotope the boundary of a tubular neighborhood around $N^3_1$, corresponding to the pre-image of an arc $A_1^*\subset M^*$ in the cycle $K^1$,  to the boundary of a tubular neighborhood around $N^3_2$,
corresponding to the pre-image of the arc opposite to $A_1^*$.


We now prove part (2), corresponding to graph  (4) in figure~\ref{F:3_vertices_wgraphs}. Recall that in this case one of the edges in  $K^1$ corresponds to orbits with  isotropy $\zzz_2$, while the other one corresponds to orbits with isotropy $\zzz_k$, $k>2$. We will denote the arc in $K^1$ corresponding to a fixed point set component of isotropy $\zzz_2$ by $A_0^*$, and we will let $A_1^*$ be the arc in $K^1$ corresponding to the fixed point set component with finite isotropy $\zzz_k$. We now form an arc $A^*_2$ in $M^*$ by joining the vertex not contained in $K^1$ to $A^*_1$ via a shortest geodesic in $M^*$. The interior of this arc consists of principal orbits and the pre-image of this arc is a cohomogeneity one $3$-manifold, $N^3_2$.  Proceeding as in cases (1) and (4), we may decompose $M^5$ as a union of disc bundles over $N^3_1=\pi^{-1}(A^*_1)$ and $N^3_2$.

To prove part (3), we let  $A_1^*$ be the arc not contained in the weighted claw, that is, the arc containing two isolated circle orbits and corresponding to a fixed point set component of isotropy $\zzz_k$, $k>2$. We let $A_2^*$ be the arc in the claw containing the  isolated circle orbit not contained in $A^*_1$. Proceeding as above,  we may decompose $M^5$ as a union of disc bundles over $N^3_1=\pi^{-1}(A^*_1)$ and $N^3_2=\pi^{-1}(A_2^*)$.

\end{proof}

We are now ready to prove proposition~\ref{p:finiteisotropy}. Our strategy will be to analyze the weighted graphs grouped into three separate cases, where the first two cases  correspond to part (1) of  proposition~\ref{p:finiteisotropy} and the third case corresponds to part (2) of proposition~\ref{p:finiteisotropy}.
 The first case will be all the graphs in figure \ref{F:3_vertices_wgraphs}, with the exception of graph (4). The second case will consist of graph (4) of figure \ref{F:3_vertices_wgraphs} and graphs (1) and (2) of figure \ref{F:claw_wgraphs}. The third case will be the graphs of figure \ref{F:4_vertices_wgraphs}.


\begin{proof}[Proof of proposition~\ref{p:finiteisotropy} - (1)]

The weighted graph corresponding to the singular set of the action is one of those listed in figure~\ref{F:3_vertices_wgraphs}  or figure~\ref{F:claw_wgraphs}. It follows from the discussion in section~\ref{SS:component_finite_iso} that if the weighted graph is one of those in figure~\ref{F:3_vertices_wgraphs}, then  a fixed point set component of finite isotropy can only be one of $S^3$, $\lpq$ or $S^2\times S^1$; if the weighted graph is one of those in figure~\ref{F:claw_wgraphs}, then the fixed point set components of finite isotropy corresponding to arcs in the claw can only be $S^2\tilde{\times} S^1$  or $\rrr P^2\times S^1$ and the corresponding isotropy subgroup is  a $\zzz_2$ subgroup of $T^2$ in each case.

As mentioned above, we have divided the proof into three cases: case A, where the graphs are all those from figure \ref{F:3_vertices_wgraphs} with the exception of graph (4); case B, corresponding to  graphs (4) from figure \ref{F:3_vertices_wgraphs} and (1) and (2) from figure \ref{F:claw_wgraphs}.
\\


\noindent \emph{Case A.} We have the following lemma.


\begin{lemma}\label{l:nolpqs2xs1} Let $T^2$ act isometrically on $M^5$, a closed, simply-connected $5$-manifold of non-negative curvature and suppose that $M^*=S^3$. If there are exactly  $3$ isolated circle orbits and the weighted graph of the action is one of graphs (1), (2) or (3) from figure \ref{F:3_vertices_wgraphs}, then  neither $\lpq$ or $S^2\times S^1$ may be a component of a fixed point set of finite isotropy.
\end{lemma}


\begin{proof}
If the weighted graph is one of graphs (1) or (2), which do not not contain a cycle, then we may complete it to a cycle,  adding edges corresponding to shortest geodesics consisting of regular points in the orbit space, so that each vertex in the graph has degree $2$ (see figure~\ref{F:3_vertices_completed_arc}). We may then decompose $M^5$ as the union of a disc bundle over a fixed point set component of finite isotropy and the remaining isolated circle orbit. A tubular neighborhood around the isolated circle orbit will be a $D^4$-bundle over $S^1$ with boundary an $S^3$ bundle over $S^1$. A tubular neighborhood around the fixed point set component of finite isotropy will be a $D^2$-bundle over $\lpq$ or $S^2\times S^1$ and therefore the boundary of both tubular neighborhoods must be $S^3\times S^1$. When we consider the Mayer-Vietoris sequence of this decomposition we immediately obtain a contradiction and therefore neither of these two manifolds may occur as a fixed point set component of finite isotropy.
 \end{proof}
 

\noindent \emph{Case B.}  We have the following lemma.

\begin{lemma}\label{l:caseB} Let $T^2$ act isometrically on $M^5$, a closed, simply-connected $5$-manifold of non-negative curvature and suppose that $M^*=S^3$, there are exactly three isolated circle orbits and the singular set corresponds to graph (4) of figure \ref{F:3_vertices_wgraphs} or graph (2) of figure \ref{F:claw_wgraphs}. Then a fixed point set component of finite isotropy $\zzz_k$, $k\neq2$, can only be one of  $S^3$ or $\rrr P^3$.
\end{lemma}

\begin{proof} We must rule out $\lpq$, where $(p,q)\neq (2,1)$, and  $S^2\times S^1$ as fixed point set components of isotropy $\zzz_k$, $k\neq 2$.
Recall that, for both graphs under consideration,  $M^5$ decomposes as a union of disc bundles over one of $\lpq$ or $S^2\times S^1$, and over one of $S^2\tilde{\times} S^1$ or $\rrr P^2\times S^1$.  It follows from the Mayer-Vietoris sequence of this decomposition that only two possibilities do not give rise to a contradiction: that $M^5$ may be the union of disc bundles over $\rrr P^3$ and $S^2\tilde{\times} S^1$ or over $S^3$ and $S^2\tilde{\times} S^1$.
\end{proof}

With these two lemmas we may now complete the proof of part (1) of proposition~\ref{p:finiteisotropy}. From
lemma \ref{l:nolpqs2xs1} above we conclude that the only possible fixed point set components
 for  graphs (1), (2), and (3) of figure \ref{F:3_vertices_wgraphs} are $S^3$. In this case, it follows  from the
 Mayer-Vietoris sequence that $H_2(M^5)=0$ and therefore $M^5$ is diffeomorphic to $S^5$ by work of Smale and Barden \cite{Sm, Ba}.

From lemma~\ref{l:caseB}  we note that for graphs (4) of figure \ref{F:3_vertices_wgraphs} or graph (2) of figure \ref{F:claw_wgraphs} that $H_2(M^5)=0$ when $\rrr P^3$ is the fixed point set component of  isotropy $\zzz_k, k\neq 2$,  and $H_2(M^5)=\zzz_2$ when $S^3$ is. Both graphs (1) and (2) of figure \ref{F:claw_wgraphs} contain a weighted claw and in the case of graph (1), we may complete the graph, joining two disjoint arcs via edges corresponding to shortest geodesics consisting of regular points in the orbit space.   We may then decompose the manifold as a union of disc bundles over the pre-image of the arc joining two edges of $\zzz_2$ isotropy and the remaining edge of $\zzz_2$ isotropy.

We further note that in these last two cases it follows from corollary \ref{c:orbitconfigs} that $H_2(M^5)=\zzz_2$ and from theorem \ref{t:cohominequalities} that the arcs corresponding to $3$-manifolds of $\zzz_2$ isotropy must be $S^2\tilde{\times} S^1$. It now follows by \cite{Sm, Ba} that for graph (4) of figure \ref{F:3_vertices_wgraphs} $M^5$ is diffeomorphic to $S^5$ or the Wu manifold and for graphs (1) and (2) of figure \ref{F:claw_wgraphs} $M^5$ is diffeomorphic to the Wu manifold.


\begin{figure}
\centering
\includegraphics[scale=0.5]{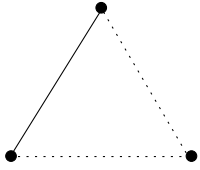}
\caption{Completing a weighted graph with three vertices to form a cycle. The solid edge corresponds to orbits with finite cyclic isotropy, while the dotted edges correspond to principal orbits.}
\label{F:3_vertices_completed_arc}
\end{figure}


We have now completed the proof of part (1) of proposition~\ref{p:finiteisotropy}.

\end{proof}



\begin{figure}
\centering
\includegraphics[scale=0.7]{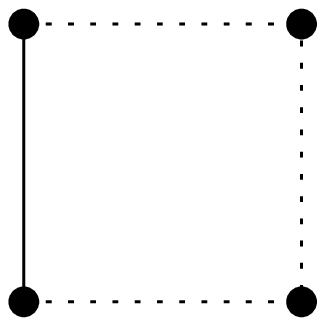}
\caption{Completing a weighted graph with edges corresponding to principal orbits to obtain a cycle. The solid edge corresponds to orbits with finite cyclic isotropy, while the dotted edges correspond to principal orbits.}
\label{F:4_vertices_completed_arc}
\end{figure}

It remains to prove (2) of proposition~\ref{p:finiteisotropy}. To do this, it suffices to show that $H_2(M^5)\cong \zzz$ for every possible fixed point set of finite isotropy and  thus,  by \cite{Sm,Ba}, $M^5$ is diffeomorphic to one of the two $S^3$ bundles over $S^2$.

\begin{proof}[Proof of proposition~\ref{p:finiteisotropy} - (2)] In this case the possible weighted graphs are shown in figure~\ref{F:4_vertices_wgraphs}. For graphs (1) through (4),  we may complete the weighted graph by joining disjoint isolated circle orbits or arcs via edges corresponding to shortest geodesics consisting of regular points in the orbit space. In this way we obtain a graph that is an unknotted cycle (see figure~\ref{F:4_vertices_completed_arc}) and now for all the possible graphs we may decompose $M^5$ as the union of two disc bundles over 
the $3$-dimensional manifolds that correspond to opposite arcs of the cycle. In this particular case, the $3$-dimensional manifold 
may be one of $S^3$, $\lpq$ or $S^2\times S^1$.
In all cases and for all possible combinations, we see that $H_2(M^5)=\zzz$ and the result follows.


\end{proof}


\section{Some examples of isometric $T^2$ actions on simply-connected, non-negatively curved $5$-manifolds}
\label{S:Examples}


\subsection{Examples of actions with codimension 2 fixed point set}

It is easy to find examples of such actions and we list a few here.


\begin{example}
Given $(\theta_1, \theta_2)\in T^2$ and $(z_1, z_2, z_3)\in S^5\subset \ccc^3$, let
\bdm
((\theta_1, \theta_2), (z_1, z_2, z_3))\longmapsto
(e^{2\pi i\theta_1}z_1, e^{2\pi i\theta_2}z_2, z_3). 
\edm
Here both circles $\theta_1$ and $\theta_2$ fix a $3$-sphere. The corresponding singular set in the orbit space is $3$ isolated singular points.
\end{example}


\begin{example}
Given $(\theta_1, \theta_2)\in T^2$ and $(z_1, z_2, x_1, x_2, x_3)\in S^3\times S^2\subset \ccc^2\times \rrr^3$, let
\bdm
((\theta_1, \theta_2), (z_1, z_2, x_1, x_2, x_3))\longmapsto
(e^{2\pi i\theta_1}z_1, e^{2\pi i\theta_2}z_2, x_1, x_2, x_3).
 \edm
 Here both circles $\theta_1$ and $\theta_2$ fix an $S^2\times S^1$ and the action is the product of the cohomogeneity one action on $S^3$ combined with the trivial action on $S^2$. The corresponding singular set in the orbit space is $4$ isolated singular points.

\end{example}


\subsection{Examples of actions with finite isotropy}
We give examples of actions on $S^5$ and on $S^3\times S^2$ with finite isotropy and with $3$ and $4$ isolated circle orbits, respectively. The action on $S^5$ was given by  Rong \cite{R} and we include it here for the sake of completeness.


\begin{example}
Given $(\theta_1, \theta_2)\in T^2$ and  $(z_1, z_2, z_3)\in S^5\subset \ccc^3$, let

\bdm
((\theta_1, \theta_2), (z_1, z_2, z_3))\longmapsto
(e^{2\pi i(\theta_1 + p\theta_2)}z_1, e^{2\pi i(\theta_1 + q\theta_2)}z_2, e^{2\pi i(\theta_1 + r\theta_2)}z_3). 
\edm
 Here there are $3$ isolated circle orbits. If $p, q, r$ are pairwise
relatively  prime and the differences $(p-q)$, $(p-r)$ and $(q-r)$ are also pairwise relatively prime, then the singular set of the action is a cycle in the orbit space and the closure of each edge corresponds to an $S^3$ fixed by finite isotropy.
\end{example}


\begin{example}
\label{EX:fixed_S3}
Given $(\theta_1, \theta_2)\in T^2$ and   $v=(z_1, z_2, x_1, x_2, x_3)\in S^3\times S^2\subset \ccc^2\times \rrr^3$, we let $(\theta_1, \theta_2)$ act on  $v$  by
\[
((\theta_1,\theta_2),v)\mapsto A(\theta_1,\theta_2)v,
\]
where $A(\theta_1,\theta_2)$ is the matrix
\[
\left( \begin{array}{ccccc}
e^{2\pi i(\theta_1 + p\theta_2)}& 0 &0&0&0\\
 0&e^{2\pi i(\theta_1 + q\theta_2)}&0&0&0\\  
 0&0& \cos(\theta_1 + r\theta_2) & \sin((\theta_1 + r\theta_2)& 0 \\
 0&0&-\sin((\theta_1 + r\theta_2) &  \cos(\theta_1 + r\theta_2) &0\\
  0&0&0&0&1\\
\end{array} \right),
\]
$p$, $q$, $r$ are pairwise relatively  prime integers, as are the differences $(p-q)$, $(p-r)$ and $(q-r)$, and, without loss of generality, $p>q>r$. Here there are $4$ isolated circle orbits and the finite groups  $\zzz_{p-r}, \zzz_{q-r}$ 
each fix a distinct $S^2\times S^1$ that has empty intersection with the other whereas the finite group $\zzz_{p-q}$ fixes two disjoint copies of $S^3$, intersecting each of the fixed $S^2\times S^1$ in an isolated circle orbit. The corresponding singular set in the orbit space is a quadrangle with vertices corresponding to isolated circle orbits and edges corresponding to arcs with finite isotropy.
\end{example}


\begin{example}
\label{EX:Wu_manifold}
Let $T^2\subset SU(3)$ act canonically on $SU(3)/SO(3)$.
There are three involutions given by  the diagonal matrices with entries $(-1, -1, 1), (-1, 1, -1)$ and 
$(1, -1, -1)$. Each of these involutions will fix an $S(U(2)\times U(1))/S(O(2)\times O(1))=S^2\tilde{\times}S^1$,
each of which intersects in a $S(U(1)\times U(1)\times U(1))/S(O(1)\times O(1)\times O(1))=T^2/(\zzz_2\times \zzz_2)$. 
The corresponding singular set in the orbit space is a weighted claw.
\end{example}


\end{document}